\documentclass[11pt]{article}    
\usepackage[margin=1in]{geometry}
\usepackage{graphicx}
\usepackage{amsmath,amsfonts,amsthm,amssymb}
\usepackage{mathrsfs}
\usepackage{hyperref}
\usepackage{caption}
\usepackage{subcaption}
\usepackage{comment}
\usepackage[shortlabels]{enumitem}

\numberwithin{equation}{section}

\newtheorem{thm}{Theorem}
\newtheorem{lem}{Lemma}[section]
\newtheorem{prop}{Proposition}

\theoremstyle{definition}
\newtheorem{defn}{Definition}

\theoremstyle{remark}
\newtheorem{rem}{Remark}[section]

\usepackage{xcolor}
\newcommand{\Red}[1]{{\color{red}#1}}
\newcommand{\Blu}[1]{}

\newcommand{\pare}[1]{\left(#1\right)}

\newcommand{\Brac}[1]{\big\{#1\big\}}
\newcommand{\Pare}[1]{\big(#1\big)}
\newcommand{\PAre}[1]{\Big(#1\Big)}

\newcommand{\abs}[1]{\left\lvert #1 \right\rvert}

\newcommand{\ABs}[1]{\Big\lvert #1 \Big\rvert}

\newcommand{\dD}[2]{{\partial #1}/{\partial #2}}

\newcommand{\AnD}{\quad\text{and}\quad}
\newcommand{\AND}{\qquad\text{and}\qquad}

\newcommand{\RR}{\mathbb{R}}
\newcommand{\Ss}{\mathbb{S}}
\newcommand{\CC}{\mathbb{C}}
\newcommand{\NN}{\mathbb{N}}
\newcommand{\ZZ}{\mathbb{Z}}

\DeclareMathOperator{\supp}{supp}

\DeclareMathOperator{\sgn}{sgn}

\newcommand{\He}{\mathbf{H}}
\newcommand{\vT}{\Theta}
\newcommand{\cS}{\mathcal{C}}
\newcommand{\rn}{\rho_n}

\newcommand{\Int}{\mathcal{I}}
\newcommand{\Cp}{\mathcal{T}}
\newcommand{\Pc}{\mathcal{S}}
\newcommand{\CP}{\mathcal{U}}
\newcommand{\N}{\mathcal{N}}
\newcommand{\Gp}{\mathcal{G}}
\newcommand{\inT}{\mathscr{I}}
\newcommand{\nid}{\not\equiv}
\begin{document}
\title{Finiteness of Nonscattering Wavenumbers for Herglotz Incident Waves}

\author{Jingni Xiao \footnote{Department of Mathematics, Drexel University, Philadelphia}}

\maketitle


\begin{abstract}
	This paper continues the study initiated in \cite{VogXia25} on nonscattering phenomena for inhomogeneous media. We investigate star-shaped domains in $\mathbb{R}^2$ and establish finiteness results for nonscattering wavenumbers associated with Herglotz incident waves of fixed density. First, for ellipses \Blu{we establish finiteness for all constant contrasts $q\neq 1$, removing the geometric restrictions required in previous work.} Second, for admissible star-shaped domains with $q\in(0,1)$, we introduce a flexible interval-wise geometric framework that unifies and generalizes earlier finiteness results. Our results reveal that infinite sequences of nonscattering wavenumbers are tied to exact radial symmetry and cannot persist under admissible geometric perturbations. 

\end{abstract}

\section{Introduction}
We consider the time-harmonic scattering problem generated by an inhomogeneous medium:
\begin{equation*}
	\begin{cases}
		\Delta u + k^2 q u =0,		&\qquad		\mbox{in $\RR^2$},
		\\
		\displaystyle\lim_{r\to\infty}\sqrt{r}\pare{\partial_ru^s-iku^s}=0,\quad r=|x|,&\qquad\mbox{uniformly for all $x/|x|\in\Ss^1$},
	\end{cases}
\end{equation*}
where $k$ is the wavenumber and $q$ is a bounded medium coefficient with compactly supported contrast $q-1$. The total field $u=u^i+u^s$ consists of an incident field $u^i$ satisfying $\Delta u^i+k^2u^i=0$ and a scattered field $u^s$ satisfying the Sommerfeld radiation condition. We assume that there exists a bounded Lipschitz domain $\Omega$ such that $\Omega^e:=\RR^2\setminus\overline{\Omega}$ is connected, $q\neq 1$ in $\Omega$, and $q=1$ in $\Omega^e$. We shall refer to $(q;\Omega)$ as the inhomogeneous medium.

We investigate the \emph{nonscattering phenomenon}, namely the existence of wavenumbers $k$ and incident fields $u^i$ for which the scattered field vanishes identically in $\Omega_e$ \cite{LukD07,AleR11,BPS14corner,Blas18,VogX21,BlLi21,CakVog23,SalS21,CaVoX23,VogXia25,KowSS24,KowLSS24}. In this case the medium $(q;\Omega)$ is said to be \emph{nonscattering} with respect to $u^i$, and $k$ is called a nonscattering wavenumber. Understanding whether such wavenumbers exist, and how they depend on the geometry of $\Omega$ and the contrast $q$, is a central problem in scattering theory and inverse scattering.

The existence of nonscattering wavenumbers depends delicately on properties of $q$ (and the geometry of its support). In radially symmetric configurations, \emph{infinite} sequences of nonscattering wavenumbers are known to exist \cite{CoM88,CCH23book}. In this case, for \emph{each} such $k$, there exist infinitely many nonscattering Herglotz incident waves of the form
\begin{equation*}\label{herglotz}
	\He [k, \phi](x)=\int_{-\pi}^{\pi}\phi(\xi) e^{ik\xi\cdot x}d\vT_{\xi},
\end{equation*}
which are referred to as the Herglotz wave functions with density $\phi$. Hereinafter, we denote $\vT_{\xi}$ as the angular coordinate of $\xi$ for any unit vector $\xi\in\Ss^1$. In contrast, domains with corners or geometric singularities generically scatter all incident waves, that is, there are \emph{no} nonscattering wavenumbers in this situation. This result is proven in \cite{BPS14corner} for convex corners of a straight angle with $q$ a ``H\"{o}lder-like'' function in $\supp (q-1)$ satisfying $q-1\neq 0$ at the corner. Sequential generalizations of this result to the cases for convex or concave corners of arbitrary angles as well as $q-1$ vanishing to higher orders at the corner can be found in, for example, \cite{PaSaVe17,ElHu18,Xiao22}; and extensions to other models including the Maxwell system, Lam\'e system, anisotropic media, conducting boundaries, and source scattering, etc., are established in, for instance, \cite{LiXi17,Blas18,BlLi19,BlVe20,BlLiXi21,CaXi21,DiCaLi21,BlLi21,BlPo22,CakHLV25,DiGeTa25}. Note that the corner scattering result in \cite{BPS14corner} is implicitly stated in an earlier work \cite{KusS03} for any convex corner but with the restriction that $(q-1)u^i\neq 0$ at the corner. More recently work on the non-existence of nonscattering waves focused on the regularity of $q$ and support of $q-1$. It is shown that regular $q$ but with irregular shape of $\supp(q-1)$ always scatter incident waves satisfying $(q-1)u^i\neq 0$ at the irregular boundary point of $\supp(q-1)$ \cite{SalS21,CakVog23,KowLSS24}. These results are later extended to the case of anisotropic media in \cite{CaVoX23,KowSS24}. For smooth media with constant $q$, it is shown in \cite{HovV25anal} incident waves satisfying certain non-degeneracy condition depending on $\Omega$ always scatters provided that $\Omega$ is piecewise real-analytic with some other admissible conditions. In \cite{HovV25plane}, it is shown that plane waves always scatter for a large class of geometry $\Omega$ with constant $q$.

The finiteness of nonscattering wavenumbers is closely related to the spectral theory of interior transmission eigenvalues. While the discreteness as well as the existence of transmission eigenvalues has been extensively studied, much less is known about the finiteness of nonscattering energies. It is proven in \cite{VogX21} that there are at most finitely many nonscattering plane incident waves with a fixed incident direction for convex domains. In \cite{VogXia25}, it is shown that given a star-shaped domain $\Omega$ and a $C^1$ density function $\phi$, there are at most finitely many wavenumbers $k$ such that $(q;\Omega)$ is nonscattering under some admissible conditions. These results are proven for the case when $q>1$. 

The present work extends the finiteness results established in \cite{VogXia25} to the regime $q\in(0,1)$ and relaxes the geometric constraints on the star-shaped domain. For ellipses, we prove finiteness for both $q>1$ and $0<q<1$, improving the results of \cite{VogXia25} and removing additional geometric restrictions in the case $q>1$. More generally, for admissible $C^2$ star-shaped domains, we establish finiteness under broader geometric configurations of the radius function. \Blu{A key new feature of the present work is that the geometric assumptions are localized, allowing different portions of the boundary to satisfy different structural conditions.} While radially symmetric media admit infinitely many nonscattering wavenumbers, we show that this behavior collapses to finiteness for a broad class of star-shaped domains. 
This establishes a geometric rigidity principle: infinite nonscattering sequences are unstable under perturbations of the boundary away from circular symmetry. 

The proofs of the main results begin with a boundary integral identity characterizing nonscattering wavenumbers. For Herglotz incident waves, this identity reduces to a family of oscillatory integrals depending on the angular parameter of the auxiliary plane wave. The core of the analysis consists of studying the large-$k$ asymptotics of these integrals. Under the admissibility assumptions, the phase function possesses only nondegenerate critical points whose contributions can be controlled via stationary phase methods. In the ellipse case, the geometry allows a more refined analysis of critical points associated with focal structure, leading to sharper finiteness results.

\subsection{Main Results}
Throughout this paper, we assume that $q$ is a positive constant distinct from $1$ in $\Omega$. Our first result concerns ellipses with one of the foci located at the origin.
\begin{thm}\label{thm:ellipse2focus}
	Let $\Omega$ be an ellipse whose one focus is located at the origin. Let 
	$\phi\in C^1(\Ss^1)$ be nontrivial. Assume that either
	\begin{enumerate}[(i)]
		\item $\phi$ is real-analytic on $\Ss^1$, or,
		\item\label{cond:ecc} the eccentricity $e$ of $\Omega$ satisfies $1/q\le 1+1/{e^2}$.
	\end{enumerate} 
	Then for the medium $(q;\Omega)$, the set of nonscattering wavenumbers with respect to the incident wave $u^i=\He[k,\phi]$ is at most finite.
\end{thm}
\begin{rem}
	We note that Theorem~\ref{thm:ellipse2focus} includes the case when $q>1$, and the result is stronger than \cite{VogXia25}. \Blu{In particular, when $q>1$, Condition \ref{cond:ecc} is automatic, removing the eccentricity restriction $e^2<\sqrt{q}/(1+\sqrt{q})$ imposed in \cite{VogXia25}. For $0<q<1$, Condition \ref{cond:ecc} also permits a substantially larger range of eccentricities than those allowed in \cite{VogXia25}.}
\end{rem}

The next several results concern star-shaped domains with $q\in(0,1)$. Similar to \cite{VogXia25}, we give the definition of admissible star-shaped domains.
\begin{defn}\label{defn:star}
	Let $q\in(0,1)$. Given a bounded domain $\Omega\subset\RR^2$, we say that $\Omega$ is $C^2$ star-shaped (with respect to the origin) if $\partial \Omega$ admits a parameterization $\{\rho(\theta)\vec{\theta}:\theta\in[0,2\pi)\}$, where $\vec{\theta}=(\cos\theta,\sin\theta)^T$ and $\rho$ is positive, $2\pi$-periodic, and $C^2$ regular, called the radius function of $\Omega$.

	We say that a star-shaped domain $\Omega$ (or its radius function $\rho$) is \emph{admissible} if $\rho$ is not constant and satisfies
	\begin{equation}\label{eq:StarSmlCond0}
		|\rn'|<\frac{\sqrt{q}}{\sqrt{1-q}}\AnD
		-\frac{\sqrt{q}}{1-\sqrt{q}}<\rn''<\frac{\sqrt{q}}{1+\sqrt{q}},
	\end{equation}
	and
	\begin{equation}\label{eq:StarSmlCond2}
		\Blu{-\frac{\varrho (\theta)\Pare{1+\varrho (\theta)}}{1-q}<\rn''(\theta)<\frac{\varrho (\theta)\Pare{1-\varrho (\theta)}}{1-q}},
	\end{equation}
	for all $t$, where
	\begin{equation*}
		\rn=\ln\rho\AND\varrho(\theta):=\sqrt{q-(1-q)\rn'^2(\theta)}.
	\end{equation*}
\end{defn}
The following result concerns Herglotz waves with \Blu{densities that are nonzero almost everywhere}.
\begin{thm}\label{thm:StarAnly}
	Let $q\in(0,1)$ be constant and let $\Omega$ be an admissible star-shaped domain. Let \Blu{$\phi\in C^1(\Ss^1)$ be nonzero almost everywhere in $\Ss^1$.} Assume one of the following holds:
	\begin{enumerate}[(i)]
		\item There exists $t_0$ such that $\rho'(t_0)\neq0$ and $\rho'(t_0)\rho'(t_0+\pi)\ge 0$,
		\item For all $t$, $|\rho'(t)|+|\rho'(t+\pi)|+ |\rho''(t)| +|\rho''(t+\pi)|>0$.
	\end{enumerate}
	Then for $(q;\Omega)$, the set of nonscattering wavenumbers with respect to  $u^i=\He[k,\phi]$ is at most finite.
\end{thm}
\begin{rem}
	Theorem~\ref{thm:StarAnly} is the analogous result to \cite[Theorem 6]{VogXia25}, for the case when $q\in(0,1)$. 
	We note that the conditions imposed for star-shaped domains in Theorem~\ref{thm:StarAnly} are very general, besides the ``admissible'' condition specified in Definition~\ref{defn:star}. For example, any $\rho$ such that $\rho'(t_0)>0\ge\rho'(t_0+\pi)$ satisfies the conditions in Case (i). This includes the case when \Blu{$\sgn\rho'(t)=\sgn\rho'(t+\pi)$ for all $t$, which in turn contains the case} when $\Omega$ is symmetric with respect to the origin, namely, when $\rho(t)=\rho(t+\pi)$ for all $t$.
\end{rem}
Next result concerns Herglotz incident wave with $C^1$ \Blu{density function} $\phi$ and the  star-shaped domain $\Omega$ satisfying certain geometric conditions.
	\begin{thm}\label{thm:main1Star}
		Let $q\in(0,1)$ be constant and let $\Omega$ be an admissible star-shaped domain with radius function $\rho$, and let $\phi\in C^1(\Ss^1)$ be nontrivial. Assume that there exists a union of countably many disjoint open intervals $\N=\cup_\iota(s_{\iota,1},s_{\iota,2})\subset(0,2\pi)+\theta_0$ for some $\theta_0\in(0,2\pi)$, viewed as subsets of $\RR/(2\pi\mathbb{Z})$, such that 
		$\overline{\N\cup(\N+\pi)}=\RR/(2\pi\mathbb{Z})$. Moreover, for each $\iota$, we assume
		\begin{equation*}
			\mbox{$|\rho'|>0$ in $(s_{\iota,1},s_{\iota,2})$,\qquad $\rho'(s_{\iota,1})=\rho'(s_{\iota,2})=\rho'(s_{\iota,1}+\pi)=\rho'(s_{\iota,2}+\pi)=0$,}
		\end{equation*}
		and one of the followings holds:
		\begin{enumerate}[(i)]
			\item\label{case:1} $\sgn\rho'(\cdot+\pi)=\sgn\rho'$ on $(s_{\iota,1},s_{\iota,2})$,
			\item\label{case:2} there exists $\epsilon>0$ such that $\rho''=0$ on either $(s_{\iota,1}+\pi,s_{\iota,1}+\pi+\epsilon)$ or $(s_{\iota,2}+\pi-\epsilon,s_{\iota,2}+\pi)$.
			\item\label{case:3} $|\rho''(s_{\iota,j})|+|\rho''(s_{\iota,j}+\pi)|> 0$ for some $j=1,2$.
		\end{enumerate}
		Then for $(q;\Omega)$, the set of nonscattering wavenumbers with respect to  $u^i=\He[k,\phi]$ is at most finite.
	\end{thm}
\begin{rem}
\end{rem}

\subsection{Preliminaries}\label{sec:PreAsym}
The proofs of Theorems~\ref{thm:ellipse2focus},~\ref{thm:StarAnly},~and~\ref{thm:main1Star} rely on a boundary integral characterization of nonscattering wavenumbers. For Herglotz incident waves, the nonscattering condition can be reformulated as a family of oscillatory integral identities on $\partial\Omega$. The problem therefore reduces to analyzing the large-wavenumber asymptotics of certain boundary oscillatory integrals. We now recall the fundamental identity and introduce the notation that will be used throughout the paper.

Following \cite{VogXia25}, if $(q;\Omega)$ is nonscattering with respect to the incident field $\He[k,\phi]$, then
\begin{equation*}
	\int_{\partial\Omega} \pare{ w\partial_{\nu} \He[k,\phi]- \He[k,\phi]\partial_{\nu} w } d\sigma(x)=0,
\end{equation*}
for every solution $w$ of $\Delta w + k^2 q w = 0$ in $\Omega$. Choosing the special solution $w(x)=e^{ik\sqrt{q}\,x\cdot\eta}$ yields
\begin{equation}\label{integralB}
	i{k} \int_{\partial\Omega}\int_{0}^{2\pi}
	\nu \cdot \pare{\xi-\sqrt{q}\eta} \phi(\xi)\,
	e^{ik \pare{\sqrt{q}\eta+\xi}\cdot x} ~d\vT_{\xi}\,d\sigma(x)
	=0,\quad\mbox{for all $\eta\in \Ss^1$}.
\end{equation}
Let $\Omega$ be a $C^2$ star-shaped domain with radius function $\rho$. Parameterizing the boundary by $y(\theta):=\rho(\theta)\vec{\theta}$, the identity \eqref{integralB} can be rewritten as an oscillatory integral of the form
\begin{equation}\label{eq:IntFinal}
	\Int_k=\Int_k(\vT_\eta):= \int_{0}^{2\pi} \int_{0}^{2\pi}
	\Psi_\eta(\theta,\vT_{\xi})
	\,\phi(\xi)\,e^{ik\,\psi_\eta(\theta,\vT_{\xi})} ~d\vT_{\xi}\,d\theta=0,\quad\mbox{for all $\eta\in \Ss^1$},
\end{equation}
where
\begin{equation}\label{eq:psi}
	\psi_\eta(\theta,\vT_{\xi})=\psi(\theta,\vT_{\xi};\vT_{\eta}):= \pare{\sqrt{q}\eta+\xi}\cdot y(\theta),\qquad
	\Psi_\eta(\theta,\vT_{\xi}):=-\pare{\sqrt{q} \,\eta-\xi }\cdot y'^{\perp}(\theta),
\end{equation}
and
\begin{equation*}
	y'^\perp=-\Pare{\vec{\theta}-\rn'\vec{\theta}^\perp}\rho.
\end{equation*}
Hereinafter, for any $x=(x_1,x_2)^T\in\mathbb{R}^2$, we denote $x^\perp := (-x_2,x_1)^T$. We denote by $\Int_k^{(N)}$ the $N$-th derivative of $\Int_k$ with respect to $\vT_{\eta}$. For example,
\begin{equation*}
	\Int_k^{(1)}(\vT_\eta)=ik\sqrt{q}\int_{0}^{2\pi} \int_{0}^{2\pi}
	\Pare{\eta^{\perp}\cdot y(\theta)\Psi_\eta(\theta,\vT_{\xi})
		+\frac{1}{k}\eta\cdot y'(\theta)}\phi(\xi)\,e^{ik\,\psi_\eta(\theta,\vT_{\xi})} ~d\vT_{\xi}\,d\theta.
\end{equation*}

Equation \eqref{integralB} implies the following orthogonality identities.
\begin{lem}\label{lem:starting}
	If $(q;\Omega)$ is nonscattering with respect to $\He[k,\phi]$ and $\Omega$ is star-shaped, then
	\begin{equation}\label{eq:intInt}
		\int_{0}^{2\pi}	\Int_k(\vT_{\eta})\varphi(\vT_\eta)d\vT_{\eta}=0\qquad\mbox{for all $\varphi\in L^\infty(0,2\pi)$},
	\end{equation}
	and
	\begin{equation}\label{eq:intdiff}
		\Int_k^{(N)}=0\qquad\mbox{for all $N\in\NN$}.
	\end{equation}
\end{lem}

We analyze the asymptotics of \eqref{eq:intInt} and \eqref{eq:intdiff} as $k\to\infty$. We recall the following stationary phase result (see, e.g., \cite{BleHan86,Evans10}).
\begin{lem}\label{lem:stationary}
	Let $\psi\in C^2(\RR^n)$ and $\omega\in C^1(\RR^n)$ be $2\pi$-periodic functions. Assume that the set of points in $[0,2\pi)^n\cap \supp \omega$ such that $\nabla\psi=0$ is finite, denoted as $\{\mathrm{x}_j:j=1,\ldots,N\}$. Suppose that $\det D^2\psi(\mathrm{x}_j)\neq 0$ for all $j=1,\ldots,N$. Then, as $k\to\infty$,
	\begin{equation*}
		\int_{(0,2\pi)^n}e^{ik\psi(\mathrm{x})}\omega(\mathrm{x})d\mathrm{x}
		=\Pare{\frac{2\pi}{k}}^{\frac{n}{2}}\text{\small$\sum_{j=1}^N
			\abs{\det D^2\psi(\mathrm{x}_j)}^{-\frac{1}{2}}$}
		e^{\frac{i\pi}4\sgn D^2\psi(\mathrm{x}_j)}e^{ik\psi(\mathrm{x}_j)}\pare{\omega(\mathrm{x}_j)+O(k^{-1})},
	\end{equation*}
	where $\sgn D^2$ is the signature of the Hessian matrix defined as the number of positive eigenvalues minus the number of negative eigenvalues of the matrix.
\end{lem}

Let us consider the asymptotics of \eqref{eq:intInt} with $\varphi\in C^1(\Ss^1)$.
We calculate that
\begin{equation*}
	\dD{\psi}{\theta}=\pare{\sqrt{q}\eta+\xi}\cdot y',\qquad
	\dD{\psi}{\vT_\xi}=\xi^{\perp}\cdot y,\AND
	\dD{\psi}{\vT_\eta}=\sqrt{q}\eta^\perp\cdot y.
\end{equation*}
Hence the stationary points of \eqref{eq:intInt}, namely, $(\theta,\vT_{\xi},\vT_{\eta})$ satisfying $$\dD{\psi}{\theta}=\dD{\psi}{\vT_\xi}=\dD{\psi}{\vT_\eta}=0,$$ are
\begin{equation*}
	(\vec{\theta},\xi,\eta)=(\vec{\theta},(-1)^{l-1}\vec{\theta},(-1)^{j-1}\vec{\theta}),\quad j,l=1,2,\qquad\mbox{with $\theta$ satisfying $\rho'(\theta)=0$}.
\end{equation*}
Moreover, the Hessian of $\psi$ at the stationary points reads
\begin{equation*}
\begin{split}
		D^2\psi&=\begin{bmatrix}
		-\pare{\sqrt{q}\eta+\xi}\cdot \vec{\theta}(1-\rn'')&\xi\cdot\vec{\theta}&\sqrt{q}\eta\cdot \vec{\theta}\\
		\xi\cdot\vec{\theta}&-\xi\cdot \vec{\theta}&0\\
		\sqrt{q}\eta\cdot \vec{\theta}&0&-\sqrt{q}\eta\cdot \vec{\theta}
	\end{bmatrix}\rho
\\&=-(-1)^{l-1}\begin{bmatrix}
	\pare{1+(-1)^{j-l}\sqrt{q}}(1-\rn'')&-1&-(-1)^{j-l}\sqrt{q}\\
	-1&1&0\\
	-(-1)^{j-l}\sqrt{q}&0&(-1)^{j-l}\sqrt{q}
\end{bmatrix}\rho,
\end{split}
\end{equation*}
where we have used the fact that
\begin{equation*}
	y'=\Pare{\rn'\vec{\theta}+\vec{\theta}^\perp}\rho
	\AND y''=-\Pare{(1-\rn'^2-\rn'')\vec{\theta}-2\rn'\vec{\theta}^\perp}\rho.
\end{equation*}
Thus at any stationary point we have
\begin{equation*}
	\det D^2\psi=-(-1)^{j-l}\sqrt{q}\Pare{1+(-1)^{j-l}\sqrt{q}}\rn''(\theta)\rho^2(\theta)	,
\end{equation*}
and $D^2\psi$ is singular if and only if $\rho''(\theta)=\rho'(\theta)=0$. Applying the method of stationary phase, we show that
\begin{prop}
	Denote $\tilde{\mathcal{C}}:=\{t\in[0,2\pi):\rho'(t)=0\}$ and $\tilde{\mathcal{C}}_0:=\{t\in\tilde{\mathcal{C}}:\rho''(\theta)\neq 0\}$. Assume 
	that \eqref{eq:intInt} holds true for a sequence of $k$'s tending to infinity. Then
	\begin{equation*}
		\phi(\vec{t}\,)=\phi(-\vec{t}\,)=0,\qquad\mbox{for any $t$ which is an isolated point in $\tilde{\mathcal{C}}_0$}.
	\end{equation*}
\end{prop}
\begin{proof}
	Given $\theta_0\in \tilde{\mathcal{C}}_0$, let $\varphi$ be a $2\pi$-periodic $C^\infty(\RR)$ function such that $\varphi(\theta_0)=1$ and $\varphi=0$ in a neighborhood of $\tilde{\mathcal{C}}\cup(\tilde{\mathcal{C}}+\pi)\backslash\{\theta_0\}$. Then the asymptotics of \eqref{eq:intInt} as $k\to \infty$ reads
	\begin{equation*}
		-\frac{(2\pi/k)^{\frac{3}{2}}}{q^{\frac{1}{4}}\abs{\rn''(\theta_0)}^{\frac{1}{2}}}
		\text{\small$\sum_{l=1}^{2}(-1)^{l-1}\frac{\pare{1-(-1)^{l-1}\sqrt{q}}}{\abs{1+(-1)^{l-1}\sqrt{q}}^{\frac{1}{2}}}
			\phi((-1)^{l-1}\vec\theta)e^{\frac{i\pi}4\sgn D^2\psi+i\pare{(-1)^{l-1}+\sqrt{q}}k}
			+O\Pare{k^{-5/2}}$},
	\end{equation*}
	where $\delta$ denotes the Kronecker delta function.
	Since \eqref{eq:intInt} is valid for $k\to\infty$, we must have
	\begin{equation*}
		\frac{\abs{1-\sqrt{q}}}{\abs{1+\sqrt{q}}^{\frac{1}{2}}}|\phi(\vec\theta_0)|
		=\frac{1+\sqrt{q}}{\abs{1-\sqrt{q}}^{\frac{1}{2}}}|\phi(-\vec\theta_0)|.
	\end{equation*}
	Hence $|\phi(\vec\theta_0)|>|\phi(-\vec\theta_0)|$ or $|\phi(\vec\theta_0)|=|\phi(-\vec\theta_0)|=0$.
	On the other hand, if we take $\varphi$ to be a $2\pi$-periodic $C^\infty(\RR)$ function satisfying $\varphi(t_0+\pi)=1$ and $\varphi=0$ in a neighborhood of $\tilde{\mathcal{C}}\cup(\tilde{\mathcal{C}}+\pi)\backslash\{t_0\pm\pi\}$. Then the asymptotics of \eqref{eq:intInt} as $k\to \infty$ becomes
	\begin{equation*}
		-\frac{(2\pi/k)^{\frac{3}{2}}}{q^{\frac{1}{4}}\abs{\rn''(t_0)}^{\frac{1}{2}}}
		\text{\small$\sum_{l=1}^{2}(-1)^{l-1}\frac{\pare{1+(-1)^{l-1}\sqrt{q}}}{\Pare{1-(-1)^{l-1}\sqrt{q}}^{\frac{1}{2}}}
			\phi((-1)^{l-1}\vec\theta_0)e^{\frac{i\pi}4\sgn D^2\psi+i\pare{(-1)^{l-1}-\sqrt{q}}k}
			+O\Pare{k^{-5/2}}$}.
	\end{equation*}
	Similar as before, we obtain that $|\phi(\vec\theta_0)|<|\phi(-\vec\theta_0)|$ or $|\phi(\vec\theta_0)|=|\phi(-\vec\theta_0)|=0$, and thus the latter must hold true.
\end{proof}

Let us now turn to the asymptotics of \eqref{eq:intdiff}. Denote $\cS=\cS_{\vT_\eta}$ as the set of stationary points $(\theta,\vT_{\xi})\in[0,2\pi)^2$ such that $\dD{\psi_\eta}{\theta}=\dD{\psi_\eta}{\vT_{\xi}}=0$. We obtain from straightforward calculation (see also \cite{VogXia25}) that $(\theta,\vT_{\xi})\in\cS_{\vT_\eta}$ if and only if
\begin{equation}\label{eq:criticalStar}
	\xi=(-1)^{l-1}\vec{\theta}\AnD
	(\sqrt{q}\eta_l\cdot\vec{\theta}+1)\rn'(\theta)=\sqrt{q}\eta_l^{\perp}\cdot\vec{\theta}
	,\qquad\mbox{for $l=1$ or $2$},
\end{equation}
where $\eta_l:=(-1)^{l-1}\eta$.
Given $\vT_{\eta}$, assume for the time being that $\cS_{\vT_\eta}$ is nonempty and finite, and that $\det D^2\psi_\eta(\theta,\vT_\xi)\neq 0$ for all $(\theta,\vT_{\xi})\in\cS_{\vT_\eta}$, where $D^2\psi_\eta$ is the Hessian
\begin{equation*}\label{eq:Hessian}
	D^2\psi_\eta(\theta,\vT_{\xi})=\begin{bmatrix}
		\pare{\sqrt{q}\,\eta+\xi}\cdot y''(\theta)& \xi^{\perp}\cdot y'(\theta)
		\\
		\xi^{\perp}\cdot y'(\theta)&-\xi\cdot y(\theta)
	\end{bmatrix}.
\end{equation*}
For $(\theta,\vT_{\xi})\in\cS_{\vT_\eta}$ with $\xi=(-1)^{l-1}\vec{\theta}$ we have
\begin{equation}\label{eq:HessianStat}
	D^2\psi_\eta
	=-(-1)^{l-1}\begin{bmatrix}
		(\sqrt{q}\eta_l\cdot\vec{\theta}+1)(1+\rn'^2-\rn'')&-1
		\\
		-1&1
	\end{bmatrix}\rho.
\end{equation}
Applying Lemma~\ref{lem:stationary} we obtain that
\begin{equation*}\label{eq:asympGeneral}
	\Int_k(\vT_\eta) = \frac{2 \pi}{k}  \text{\small$ \sum_{(\theta,\vT_\xi)\in\cS}
		\frac{\Psi_\eta e^{\frac{i\pi}4\sgn D^2\psi_\eta+ik\psi_\eta}}{\abs{\det D^2\psi_\eta}^{1/2}}|_{(\theta,\vT_\xi)}\phi(\xi)$}
	+O\Pare{\frac{1}{k^2}},\qquad\mbox{as $k\to\infty$}.
\end{equation*}
Similarly, for the derivatives of $\Int_k$ we have for all $N\in\NN$,
\begin{equation}\label{eq:asympGen_diffeta}
	\text{\small$\Int_k^{(N)}(\vT_\eta)=
	\frac{2\pi}{k}\pare{ik\sqrt{q}}^N \sum_{(\theta,\vT_\xi)\in\cS}
			\frac{\Psi_\eta e^{\frac{i\pi}4\sgn D^2\psi_\eta+ik\psi_\eta}}{\abs{\det D^2\psi_\eta}^{1/2}}|_{(\theta,\vT_\xi)}
		\Pare{\eta^{\perp}\cdot y(\theta)}^N\phi(\xi)
		+O(k^{N-2})$},
\end{equation}
as $k\to\infty$.
We will also need the following two results. The first is proven in \cite{VogXia25}; the second is a direct consequence of properties of transmission eigenvalues, see \cite{CCH23book}.
\begin{lem}\label{lem:Psi=0}
Given $\eta \in \Ss^1$, if $\theta,\vT_\xi$ satisfy $\dD{\psi_{\eta}}{\theta}=0$, then $\Psi_{\eta}(\theta,\vT_\xi)\neq 0$.
\end{lem}
\begin{lem}\label{lem:kinfty}
	Given a bounded medium $(q;\Omega)$ satisfying $q>1$ in $\Omega$ or $0<q<1$ in $\Omega$, suppose that there are infinitely many wavenumbers $k$ such that $(q;\Omega)$ is nonscattering with respect to some incident wave $u^i_k$ with wavenumber $k$. Then the wavenumbers $k$'s accumulate at and only at infinity.
\end{lem}
\section{Ellipses}
Let $\Omega$ be an ellipse with one of its foci located at the origin. Up to a rotational change of coordinates we can express $\Omega$ as a star-shaped domain with the radius function
\begin{equation}\label{eq:rhoEllip}
	\rho(t)=\frac{a(1-e^2)}{1+e\cos t},
\end{equation}
where $e\in(0,1)$ is the eccentricity of the ellipse and $a$ is half of the major diameter. In the limiting case when $e=0$, the ellipse degenerate into a disk of radius $a$.
A direct computation yields
\begin{equation*}
	\rn'(t)=\frac{e\sin t}{1+e\cos t}
\AND \rn''( t)=e\frac{e+\cos t}{\pare{1+e\cos t}^{2}}.
\end{equation*}
We prove Theorem~\ref{thm:ellipse2focus} by considering the asymptotics \eqref{eq:asympGen_diffeta}. To that end, we need to investigate some properties of the stationary points of the phase function in \eqref{eq:IntFinal}.

Given $\eta\in\Ss^1$, recall that $(\theta,\vT_{\xi})\in\cS_{\vT_\eta}$ if and only if \eqref{eq:criticalStar} is satisfied.
The second equation in \eqref{eq:criticalStar} is equivalent to
\begin{equation}\label{eq:criticEqellip}
	\rn'(\theta)=\frac{e\sin \theta}{1+e\cos \theta}=\frac{\sqrt{q}\sin (\theta-\vT_{\eta_l})}{1+\sqrt{q}\cos (\theta-\vT_{\eta_l})},
	\qquad\mbox{where $\eta_l=(-1)^{l-1}\eta$},
\end{equation}
which can be also written as
\begin{equation}\label{eq:criticEllip}
	e\sin\theta+e\sqrt{q}\sin\vT_{\eta_l}=\sqrt{q}\sin(\theta-\vT_{\eta_l}).
\end{equation}
We first observe that every $\theta\in\RR$ is a solution to \eqref{eq:criticEqellip} and \eqref{eq:criticEllip} in the case when
\begin{equation}\label{eq:except}
	q=e^2\AnD \eta_l=0.
\end{equation}
In all other cases we can solve \eqref{eq:criticEqellip} or \eqref{eq:criticEllip} straightforwardly and obtain two solutions $\vec\theta_{j,l}\in\Ss^1$, $j=1,2$, for each $\eta\in\Ss^1$ and each $l=1,2$. They are given by
\begin{equation}\label{eq:cosjl}
	\cos\theta_{j,l}=\frac{-\sin^2\vT_{\eta_l}+(-1)^{j-1}\Pare{\cos\vT_{\eta_l}-\frac{e}{\sqrt{q}}}A(\eta_l)}{eB(\eta_l)},
\end{equation}
and
\begin{equation}\label{eq:sinjl}
	\sin\theta_{j,l}=\sin \vT_{\eta_l} \frac{\cos\vT_{\eta_l}-\frac{e}{\sqrt{q}}+(-1)^{j-1}A(\eta_l)}{eB(\eta_l)},
\end{equation}
except when $q=e^2$ and $\eta_l=0$ are satisfied simultaneously.
The quantities $A$ and $B$ in \eqref{eq:cosjl} and \eqref{eq:sinjl} are given by
\begin{equation*}
	A(\eta_l):=\sqrt{B(\eta_l)-\sin^2\vT_{\eta_l}}
	\AnD B(\eta_l):=\frac{1}{e^2}+\frac{1}{q}-\frac{2\cos\vT_{\eta_l}}{e\sqrt{q}}.
\end{equation*}
It is observed that $A(\eta_l)=0$ if and only if \eqref{eq:except} is satisfied. 
We can also calculate that
\begin{equation}
	\eta_l\cdot\vec{\theta}_{j,l}=\cos(\theta_{j,l}-\vT_{\eta_l})
	=\frac{-\sin^2\vT_{\eta_l}+(-1)^{j-1}\Pare{\frac{\sqrt{q}}{e}-\cos\vT_{\eta_l}}A(\eta_l)}{\sqrt{q}B(\eta_l)},
\end{equation}
and
\begin{equation}\label{eq:sint-eta}
	\eta_l^\perp\cdot\vec{\theta}_{j,l}=\sin(\theta_{j,l}-\vT_{\eta_l})
	=\sin \vT_{\eta_l} \frac{\frac{\sqrt{q}}{e}-\cos\vT_{\eta_l}+(-1)^{j-1}A(\eta_l)}{\sqrt{q}B(\eta_l)}.
\end{equation}
Notice that
\begin{equation*}
	1+\rn'^2(t)-\rn''(t)=\frac{1}{1+e\cos t}.
\end{equation*}
Recalling \eqref{eq:HessianStat}, at the stationary points $(\theta_{j,l},\vT_{\xi})\in\cS_{\vT_\eta}$ with $\xi=(-1)^l\vec{\theta}_{j,l}$ we have
\begin{equation*}
	-(-1)^{l-1}\frac{D^2\psi_\eta}{\rho(\theta_{j,l})}=\begin{bmatrix}
		\frac{1+\sqrt{q}\cos(\theta_{j,l}-\vT_{\eta_l})}{1+e\cos\theta_{j,l}}&-1
		\\
		-1&1
	\end{bmatrix},
\end{equation*}
and hence
\begin{equation*}
	\begin{split}
		\frac{\det D^2\psi_\eta}{\rho^2(\theta_{j,l})} &=\frac{\sqrt{q}\cos(\theta_{j,l}-\vT_{\eta_l})-e\cos\theta_{j,l}}{1+e\cos\theta_{j,l}}
	=\frac{(-1)^{j-1}e\sqrt{q}}{1+e\cos\theta_{j,l}}A(\eta_l),
	\end{split}
\end{equation*}
which is zero if and only if $\eta_l=0$ and $e^2=q$.

Consider the stationary points $\theta_{j,l}=\theta_{j,l}(\vT_{\eta})$ as functions of $\vT_{\eta}$. We first notice that
\begin{equation}\label{eq:l+1}
	\theta_{j,l}(\vT_\eta+\pi)=\theta_{j,l+1}(\vT_\eta).
\end{equation}
Moreover, regarding the stationary point $\theta_{j,l}$ as a functions of $\vT_\eta$, we can differentiate \eqref{eq:criticEllip} with respect to $\vT_{\eta}$. By doing so we obtain
\begin{equation*}
	\begin{split}
		\theta_{j,l}'&=\sqrt{q}\frac{\cos(\theta_{j,l}-\vT_{\eta_l})+e\cos\vT_{\eta_l}}{\sqrt{q}\cos(\theta_{j,l}-\vT_{\eta_l})-e\cos\theta_{j,l}}
	=\frac{\cos(\theta_{j,l}-\vT_{\eta_l})+e\cos\vT_{\eta_l}}{(-1)^{j-1}eA(\eta_l)}
	\\&=\frac{\Pare{\frac{\sqrt{q}}{e}-\cos\vT_{\eta_l}}\PAre{A(\eta_l)+(-1)^{j-1}\Pare{\cos\vT_{\eta_l}-\frac{e}{\sqrt{q}}}}}{e\sqrt{q}A(\eta_l)B(\eta_l)}.
	\end{split}
\end{equation*}
Notice that
\begin{equation*}
	A^2=\Pare{\frac{1}{e^2}-1}\sin^2\vT_{\eta_l}+\frac{1}{e^2}\Pare{\cos\vT_{\eta_l}-\frac{e}{\sqrt{q}}}^2.
\end{equation*}
Therefore we can conclude that
\begin{equation}\label{eq:sgnTheta'}
	\sgn \theta_{j,l}'(\vT_{\eta})=\sgn \Pare{\frac{\sqrt{q}}{e}-\cos\vT_{\eta_l}}.
\end{equation}
Similarly, recalling \eqref{eq:sinjl} we also observe that
\begin{equation*}
	\sgn \sin \theta_{j,l}=(-1)^{j-l}\sgn \sin\vT_\eta.
\end{equation*}

The following lemma provide the ranges of $\theta_{j,l}$, $j,l=1,2$, as functions of $\vT_\eta$. Throughout this paper, we consider $\arccos$ as the bijection from $[-1,1]$ to $[0,\pi]$.
\begin{lem}\label{lem:thetajlRange}
	If $q>e^2$, then $\theta_{j,l}$ is a bijection onto $[0,2\pi)$, for each $j,l=1,2$.
	If $q<e^2$, then for each $j,l=1,2$, the range of $\theta_{j,l}$ on $[0,2\pi)$ is $$\textstyle[(-1)^{j-1}\Pare{\pi-\arccos(\frac{q}{e}+(-1)^{j-1}c_{q,e})},\pi+(-1)^{j-1}\arccos(\frac{q}{e}+(-1)^{j-1}c_{q,e})],$$
	where $c_{q,e}:=\sqrt{(1-q)(1-q/e^2)}$.
	If $q=e^2$, then for each $j=1,2$,
	$$\theta_{j,1}:(0,2\pi)\to\Pare{(-1)^{j-1}(\pi-\arccos e),\pi+(-1)^{j-1}\arccos e},$$
	$$\theta_{j,2}:(-\pi,\pi)\to\Pare{(-1)^{j-1}(\pi-\arccos e),\pi+(-1)^{j-1}\arccos e},$$
	are strictly increasing smooth bijections.
\end{lem}
\begin{proof}
	Notice that $A(t)B(t)\neq 0$ for any $t$, provided $e^2\neq q$. Hence in this case $\theta_{j,l}$, $j,l=1,2$, defined via \eqref{eq:cosjl} and \eqref{eq:sinjl} are smooth functions of $\vT_\eta$. Similarly, when $e^2= q$, for each $j,l=1,2$, $\theta_{j,l}$ is smooth away from $\vT_\eta$ such that $\vT_{\eta_l}/(2\pi)\in\ZZ$.

	If $q>e^2$, we have from \eqref{eq:sgnTheta'} that $\theta_{j,l}'>0$. The statement for this case can be then concluded by observing that $\theta_{1,1}(0)=\theta_{2,2}(0)=0=\theta_{1,2}(-\pi)=\theta_{2,1}(-\pi)$, $\theta_{1,1}(\pi)=\theta_{2,2}(\pi)=\pi=\theta_{1,2}(0)=\theta_{2,1}(0)$(, and $\theta_{1,1}(2\pi)=\theta_{2,2}(2\pi)=2\pi=\theta_{1,2}(\pi)=\theta_{2,1}(\pi)$).

	For the case when $q<e^2$, we first obtain from \eqref{eq:sgnTheta'} that $\theta_{j,1}$ increases in $(\arccos\frac{\sqrt{q}}{e},2\pi-\arccos\frac{\sqrt{q}}{e})$ and decreases in $(2\pi-\arccos\frac{\sqrt{q}}{e},2\pi+\arccos\frac{\sqrt{q}}{e})$. 
	The corresponding statement can be then concluded by recalling \eqref{eq:l+1} and calculating that
	$$\textstyle\theta_{j,1}(\arccos\frac{\sqrt{q}}{e})=(-1)^{j-1}\Pare{\pi-\arccos(\frac{q}{e}+(-1)^{j-1}\sqrt{(1-q)(1-\frac{q}{e^2})})},$$
	and
	$$\textstyle\theta_{j,1}(2\pi-\arccos\frac{\sqrt{q}}{e})=\pi+(-1)^{j-1}\arccos(\frac{q}{e}+(-1)^{j-1}\sqrt{(1-q)(1-\frac{q}{e^2})}).$$

	Finally, assume that $q=e^2$. Then $\theta_{j,l}$ has singularity when $\vT_{\eta_l}=0$, but $\theta_{j,1}$ is smooth and monotonically increasing on $(0,2\pi)$, and so is $\theta_{j,2}$ on $(-\pi,\pi)$.
	In addition, when $\vT_{\eta_l}\neq0$ we have
	\begin{equation*}
		B_l:=B(\eta_l)=\frac{2}{e^2}(1-\cos\vT_{\eta_l})=\frac{4}{e^2}\sin^2\frac{\vT_{\eta_l}}{2},\quad
		A_l:=A(\eta_l)=\frac{2}{e}|\sin\frac{\vT_{\eta_l}}{2}|\sqrt{1-e^2\cos^2\frac{\vT_{\eta_l}}{2}},
	\end{equation*}
	and hence
	\begin{equation*}
		\begin{split}
			\cos\theta_{j,l}&=\frac{-4\sin^2\frac{\vT_{\eta_l}}{2}\cos^2\frac{\vT_{\eta_l}}{2}-(-1)^{j-1}2\sin^2\frac{\vT_{\eta_l}}{2}A_l}{eB_l}
		\\&=-e\cos^2\frac{\vT_{\eta_l}}{2}-(-1)^{j-1}|\sin\frac{\vT_{\eta_l}}{2}|\sqrt{1-e^2\cos^2\frac{\vT_{\eta_l}}{2}},
		\\
			\sin\theta_{j,l}&=\sin\vT_{\eta_l} \frac{-2\sin^2\frac{\vT_{\eta_l}}{2}+(-1)^{j-1}A_l}{eB_l}
		\\&=-\frac{e}{2}\sin\vT_{\eta_l}+(-1)^{j-1}(\sgn\sin\frac{\vT_{\eta_l}}{2})\cos\frac{\vT_{\eta_l}}{2}\sqrt{1-e^2\cos^2\frac{\vT_{\eta_l}}{2}}.
		\end{split}
	\end{equation*}
	Therefore, we obtain that
	\begin{equation*}
		\lim_{\vT_\eta\to 0_+}\theta_{j,1}(\vT_\eta) =(-1)^{j-1}(\pi-\arccos e) =\lim_{\vT_\eta\to -\pi_+}\theta_{j,2}(\vT_\eta),
	\end{equation*}
	and
	\begin{equation*}
		\lim_{\vT_\eta\to 2\pi_-}\theta_{j,1}(\vT_\eta)=\pi+(-1)^{j-1}\arccos e=\lim_{\vT_\eta\to \pi_-}\theta_{j,2}(\vT_\eta).
	\end{equation*}
	The proof is complete.
\end{proof}
Define
\begin{equation*}\label{eq:fjlStar}
	f_{j,l}(\vT_\eta)=\eta^\perp \cdot y(\theta_{j,l}) =\rho(\theta_{j,l}) \sin(\theta_{j,l}-\vT_{\eta}),\qquad j,l=1,2.
\end{equation*}
The following result shows that $\#\{f_{j,l}(t): j,l=1,2\}=4$ for almost all $\vT_\eta$, where $\#$ denotes the cardinality of a set.
\begin{lem}\label{lem:fjlEllip}
	For every $t$ with $\sin t\neq 0$, the set $\{f_{j,l}(t): j,l=1,2\}$ has exactly four distinct elements, except when $\cos t= 0$ we have $f_{1,2}(t)=f_{1,1}(t)\neq f_{2,1}(t)=f_{2,2}(t)$, and when $e^2q>1$ and $\cos t= (-1)^{\iota-1}/(e\sqrt{q})$ we have $\#\{f_{j,l}(t): j,l=1,2\}=3$ with $f_{1,\iota}(t)=f_{2,\iota}(t)$.
\end{lem}
\begin{rem}
	If $\sin t=0$, we immediately obtain from \eqref{eq:sint-eta} that $f_{j,l}(t)=0$, $j,l=1,2$.
\end{rem}
\begin{proof}
	Direct calculation yields
	\begin{equation*}
		f_{j,l}(\vT_\eta)=\eta^\perp\cdot \vec{\theta}_{j,l}\rho(\theta_{j,l})
		=\frac{a(1-e^2)\sin\vT_\eta}{\sqrt{q}} \frac{\frac{\sqrt{q}}{e}-\cos\vT_{\eta_l}+(-1)^{j-1}A_l}{(1+e\cos\theta_{j,l})B_l},
	\end{equation*}
	where $A_l:=A(\eta_l)$ and $B_l:=B(\eta_l)$.
	Notice that
	\begin{equation*}
		(1+e\cos\theta_{j,l})B_l =\Pare{A_l+(-1)^{j-1}(\cos\vT_{\eta_l}-\frac{e}{\sqrt{q}})}A_l,
	\end{equation*}
	and that
	\begin{equation*}
		\begin{split}
			&\frac{A_l+(-1)^{j-1}(\frac{\sqrt{q}}{e}-\cos\vT_{\eta_l})}{A_l+(-1)^{j-1}(\cos\vT_{\eta_l}-\frac{e}{\sqrt{q}})}
		=\frac{\Pare{A_l+(-1)^{j-1}(\frac{\sqrt{q}}{e}-\cos\vT_{\eta_l})}\Pare{A_l-(-1)^{j-1}(\cos\vT_{\eta_l}-\frac{e}{\sqrt{q}})}}{A_l^2-(\cos\vT_{\eta_l}-\frac{e}{\sqrt{q}})^2}
		\\&\qquad=\frac{B_l-\sin^2\vT_{\eta_l}+(-1)^{j-1}e\sqrt{q}B_lA_l-(\frac{\sqrt{q}}{e}-\cos\vT_{\eta_l})(\cos\vT_{\eta_l}-\frac{e}{\sqrt{q}})}{B_l-\sin^2\vT_{\eta_l}-(\cos\vT_{\eta_l}-\frac{e}{\sqrt{q}})^2}
		\\&\qquad=\frac{(-1)^{j-1}e\sqrt{q}}{1-e^2}\Pare{A_l-(-1)^{j-1}(\cos\vT_{\eta_l}-\frac{1}{e\sqrt{q}})}.
		\end{split}
	\end{equation*}
	Assume that $\sin\vT_\eta\neq 0$. Then
	\begin{equation*}
		\frac{f_{j,l}(\vT_\eta)}{ea\sin\vT_\eta}-1
		= (-1)^{j-1}\frac{-(-1)^{l-1}\cos\vT_{\eta}+\frac{1}{e\sqrt{q}}}{A_l}.
	\end{equation*}
	It follows that $f_{\eta}^{1,l}=f_{\eta}^{2,l}$ if and only if $(-1)^{l-1}\cos\vT_\eta=1/(e\sqrt{q})<1$, for each $l=1,2$.
	We also observe that $f_{1,2}(\vT_\eta)-f_{2,1}(\vT_\eta)=f_{1,1}(\vT_\eta)-f_{2,2}(\vT_\eta)$. In addition, when $\cos\vT_\eta=0$ we have $A_1=A_2$ and $f_{1,1}(\vT_\eta)=f_{1,2}(\vT_\eta)\neq f_{2,1}(\vT_\eta)=f_{2,2}(\vT_\eta)$.
	Finally let us consider the case when $\sin\vT_\eta \cos\vT_\eta\neq 0$. In this case
	\begin{equation*}
		\begin{split}
			\frac{f_{1,1}(\vT_\eta)-f_{2,2}(\vT_\eta)}{ea\sin\vT_\eta}
		&=\frac{\cos\vT_{\eta}+\frac{1}{e\sqrt{q}}}{A_2}-\frac{\cos\vT_{\eta}-\frac{1}{e\sqrt{q}}}{A_1}
		=\frac{A_1^2(\cos\vT_{\eta}+\frac{1}{e\sqrt{q}})^2-A_2^2(\cos\vT_{\eta}-\frac{1}{e\sqrt{q}})^2}{A_1A_2\Pare{A_1(\cos\vT_{\eta}+\frac{1}{e\sqrt{q}})+A_2(\cos\vT_{\eta}-\frac{1}{e\sqrt{q}})}}
		\\&=\frac{4(\frac{1}{e^2}-1)(1-\frac{1}{q})\cos\vT_\eta}{e\sqrt{q}A_1A_2\Pare{A_1(\cos\vT_{\eta}+\frac{1}{e\sqrt{q}})+A_2(\cos\vT_{\eta}-\frac{1}{e\sqrt{q}})}}\neq 0,
		\end{split}
	\end{equation*}
	and similarly
	\begin{equation*}
		\begin{split}
			\frac{f_{1,2}(\vT_\eta)-f_{1,1}(\vT_\eta)}{ea\sin\vT_\eta}
			=\frac{4(\frac{1}{e^2}-1)(1-\frac{1}{q})\cos\vT_\eta}{e\sqrt{q}A_1A_2\Pare{A_1(\cos\vT_{\eta}+\frac{1}{e\sqrt{q}})-A_2(\cos\vT_{\eta}-\frac{1}{e\sqrt{q}})}}\neq 0.
		\end{split}
	\end{equation*}
	The proof is complete.
\end{proof}
Notice that $\#\{f_{j,l}(\vT_\eta): j,l=1,2\}=4$ implies $\#\{\theta^{j,l}: j,l=1,2\}=4$. Hence we have from \eqref{eq:asympGen_diffeta} that, for all $\eta\in\Ss^1$ with $\sin\vT_\eta \cos\vT_\eta\neq 0$ and $qe^2\cos^2\vT_\eta\neq 1$,
\begin{equation}\label{eq:IntNasymp}
	\text{\small$\Int_k^{(N)}(\vT_\eta)=
		\frac{2\pi}{k}\pare{ik\sqrt{q}}^N \sum_{j,l=1}^{2}
		\frac{\Psi_\eta e^{\frac{i\pi}4\sgn D^2\psi_\eta+ik\psi_\eta}}{\abs{\det D^2\psi_\eta}^{1/2}}|_{(\theta_{j,l},\theta_{j,l}+\delta_{l,2}\pi)}
		\phi((-1)^{l-1}\vec\theta_{j,l})f_{j,l}^N(\vT_\eta)
		+O(k^{N-2})$},
\end{equation}
as $k\to \infty$. We are ready to prove the following result.
\begin{prop}\label{prop:ellip}
	Let $\phi\in C^1(\Ss^1)$. Suppose that there exist a sequence of wavenumbers $k$'s tending to infinity such that $\Int_k(\vT_\eta)=0$ for all $\vT_\eta$, where $\Int_k$ is defined in \eqref{eq:IntFinal} with the associated radius function $\rho$ given by \eqref{eq:rhoEllip}. If, either $\phi$ is real-analytic, or $q$ and $e$ satisfy $1/q\le 1+1/e^2$, then $\phi$ must be identically zero.
\end{prop}
\begin{proof}
	Since $\Int_k$ is smooth in $\vT_\eta$, we have that $\Int_k^{(N)}=0$ for the sequence of wavenumbers $k$'s. \Blu{Taking $N=0,1,2,3$ in \eqref{eq:IntNasymp} yields a homogeneous linear system for the four leading-order coefficients associated with the stationary points. We} can then consider an algebraic equation with a $4\times 4$ Vandermonde coefficient matrix whose entries are generated by $f_{j,l}^N(\vT_\eta)$, $j,l=1,2$, $N=0,\ldots,3$. More precisely, we must have
	\begin{equation*}
		\begin{bmatrix}
			1&1&1&1\\a_1&a_2&a_3&a_4\\a_1^2&a_2^2&a_3^2&a_4^2\\a_1^3&a_2^3&a_3^3&a_4^3
		\end{bmatrix}
		\begin{bmatrix}
			z_1\\z_2\\z_3\\z_4
		\end{bmatrix}
		= 0,
	\end{equation*}
	where $\{a_j:j=1,\ldots,4\}=\{f_{j,l}(\vT_\eta):j,l=1,2\}$,  $$\{z_j:j=1,\ldots,4\}=\{\phi(\xi)\Psi_\eta(\theta,\vT_\xi)/\abs{\det D^2\psi_\eta(\theta,\vT_\xi)}^{1/2}\omega_{j,l}	:(\theta,\vT_\xi)=(\theta_{j,l},\theta_{j,l}+\delta_{l,2}\pi),j,l=1,2\},$$
	and $\omega_{j,l}\in\CC$ with $|\omega_{j,l}|=1$ is a limit of $e^{\frac{i\pi}4\sgn D^2\psi_\eta+ik\psi_\eta}$ as $k\to \infty$ up to a subsequence.
	Recall from Lemma~\ref{lem:Psi=0} that (see also, \cite[Lemma 3.4]{VogXia25}) $\Psi_\eta(\theta_{j,l},\theta_{j,l}+\delta_{l,2}\pi)\neq 0$. By Lemma~\ref{lem:fjlEllip} we can then conclude that the Vandermonde matrix is nonsingular and hence
\begin{equation*}
	\phi(\xi)=\phi((-1)^{l-1}\vec\theta_{j,l}(\vT_\eta))=0,\qquad j,l=1,2,
\end{equation*}
for all $\eta\in\Ss^1$ with $\sin\vT_\eta \cos\vT_\eta\neq 0$ and $qe^2\cos^2\vT_\eta\neq 1$, and hence by continuity for all $\eta\in\Ss^1$.

If $\phi$ is real-analytic on $\Ss^1$, we readily have $\phi=0$ since it vanishes on a nonempty open subset of $\Ss^1$. For the case when $1/q\le 1+1/e^2$, we claim that $\cup_{j,l=1}^2(-1)^{l-1}\vec{\theta}_{j,l}[0,2\pi)=\Ss^1$, which then implies $\phi=0$ on $\Ss^1$. In fact, if $q>e^2$, we obtain from Lemma~\ref{lem:thetajlRange} that $\vec{\theta}_{1,1}[0,2\pi)=\Ss^1$. Suppose that $e^2>q\ge 1-q/e^2$. Then $q/e-c_{q,e}\ge 0$ and hence $0<\arccos(q/e-c_{q,e})\le \pi/2$, where $c_{q,e}=\sqrt{(1-q)(1-q/e^2)}$. Therefore by Lemma~\ref{lem:thetajlRange} we have $[-\pi/2,\pi/2]\subset\theta_{2,1}[0,2\pi)$ and $[\pi/2,3\pi/2]\subset \pi+\theta_{2,2}[0,2\pi)$. 
Similarly, when $q=e^2$, the range of $\theta_{2,1}$ is $(\arccos e-\pi,\pi-\arccos e)\supset [-\pi/2,\pi/2]$, and that of $\theta_{2,2}+\pi$ is $(\arccos e,2\pi-\arccos e)\supset [\pi/2,3\pi/2]$.

The proof is complete.
\end{proof}
Theorem~\ref{thm:ellipse2focus} is a direct consequence of Proposition~\ref{prop:ellip} and Lemmas~\ref{lem:starting}~and~\ref{lem:kinfty}.

\section{Star-Shaped Domain}
In this section, we consider the case when $\Omega$ is a general $C^2$ star-shaped domain as in Definition~\ref{defn:star} and prove Theorems~\ref{thm:StarAnly}~and~\ref{thm:main1Star}. We will adopt the main ideas for the proof of Theorem~\ref{thm:ellipse2focus}. However, due to the implicit parameterization of $\partial\Omega$, more dedicated analysis is needed. 
We assume $q\in(0,1)$ throughout this section.
\subsection{Stationary Points}
We first establish the existence and properties of the stationary points $(\theta,\vT_{\xi})\in\cS_{\vT_\eta}$ associated with the phase function $\psi_{\eta}$, that is, the solutions to \eqref{eq:criticalStar}. The second equation in  \eqref{eq:criticalStar} is equivalent to
\begin{equation}\label{eq:criticalStar2}
	\rn'(\theta)=h(\theta-\vT_{\eta_l}), \qquad \text{where}\quad
	h(t):=\frac{\sqrt{q}\sin t}{1+\sqrt{q}\cos t}.
\end{equation}
We calculate that
\begin{equation*}\label{eq:hdiffstar}
	h'(t)=\sqrt{q}\frac{\sqrt{q}+\cos t}{(1+\sqrt{q}\cos t)^2}
	\AND h''(t)=\sqrt{q}\sin t\frac{2q-1+\sqrt{q}\cos t}{(1+\sqrt{q}\cos t)^3}.
\end{equation*}
\begin{lem}\label{lem:4criticStarSml}
	Denote $\theta_q=\arccos\sqrt{q}\in(0,\pi/2)$ and $\tilde{\theta}_q=\pi-\theta_q\in(\pi/2,\pi)$. Assume that there exist $\theta_1\in(0,\tilde{\theta}_q)$ and $\theta_2\in(\tilde{\theta}_q,\pi)$ such that
	\begin{equation}\label{eq:StarSmlCond1}
		|\rn'(t)|<\min_{j=1,2}h(\theta_j)\AND
		h'(\theta_2)<\rn''(t)<\min\{h'(0),h'(\theta_1)\},\qquad\mbox{for all $t$}.
	\end{equation}
	Then for each $l=1,2$ and each $\eta\in\Ss^1$, there are exactly two solutions $\theta=\Cp_{1,l}\vT_\eta$ and $\theta=\Cp_{2,l}\vT_\eta$ to \eqref{eq:criticalStar2}, which satisfy $\Cp_{1,l}\vT_\eta-\vT_{\eta_l}\in(-\theta_1,\theta_1)\subset(-\tilde{\theta}_q,\tilde{\theta}_q)$ and $\Cp_{2,l}\vT_\eta-\vT_{\eta_l}\in(\theta_2,2\pi-\theta_2)\subset(\tilde{\theta}_q,2\pi-\tilde{\theta}_q)$.
\end{lem}
\begin{proof}
	Notice that $h$ strictly decreases in the interval $[-\tilde{\theta}_q,\tilde{\theta}_q]$ and increases $[\tilde{\theta}_q,2\pi-\tilde{\theta}_q]$. In addition, $h'$ strictly decreases in $[0,\pi]$ when $q\in(0,1/4]$, and increases in $[0,{\theta}_{\tilde{q}}]$ and decreases in $[{\theta}_{\tilde{q}},\pi]$ if $q\in(1/4,1)$. In the latter case, $\theta_{\tilde{q}}:=\arccos(2\sqrt{q}-1/\sqrt{q})\in(0,\tilde{\theta}_q)$. Moreover, $h'(t)$ is symmetric with respect to $t=\pi$.
	\begin{figure}[h]
		\begin{subfigure}{.33\textwidth}
			\centering\includegraphics[scale=1]{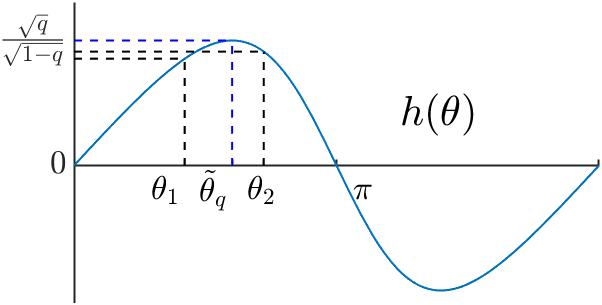}
		\end{subfigure}
		\begin{subfigure}{.33\textwidth}
			\centering\includegraphics[scale=1]{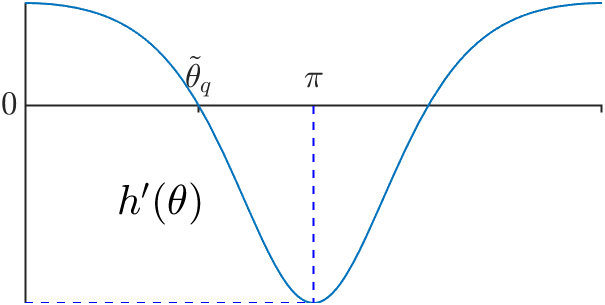}
		\end{subfigure}
		\begin{subfigure}{.33\textwidth}
			\centering\includegraphics[scale=1]{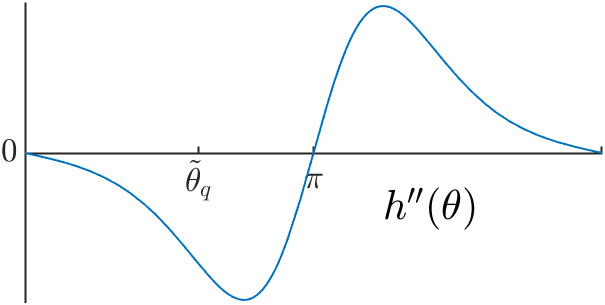}
		\end{subfigure}
		\caption{Graphs of $h$, $h'$, and $h''$, when $q=0.1$ ($q\le1/4$).}
	\end{figure}
	\begin{figure}[h]
		\begin{subfigure}{.33\textwidth}
			\centering\includegraphics[scale=1]{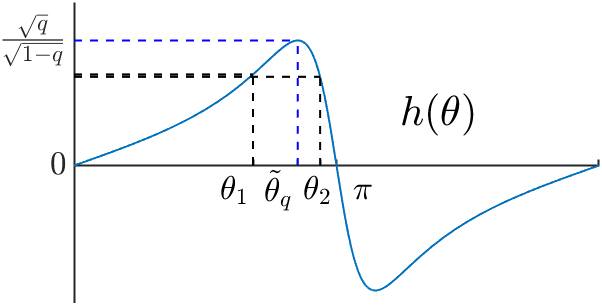}
		\end{subfigure}
		\begin{subfigure}{.33\textwidth}
			\centering\includegraphics[scale=1]{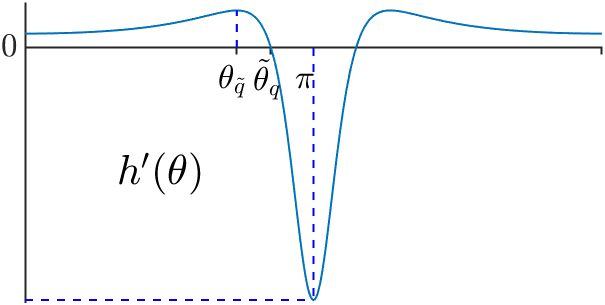}
		\end{subfigure}
		\begin{subfigure}{.33\textwidth}
			\centering\includegraphics[scale=1]{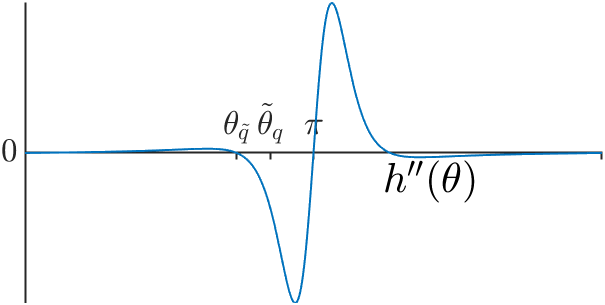}
		\end{subfigure}
		\caption{Graphs of $h$, $h'$, and $h''$, when $q=0.8$ ($1>q>1/4$).}
	\end{figure}

	Given $l=1,2$ and $\eta\in\Ss^1$, we rewrite \eqref{eq:StarSmlCond1} as $f(t):=h(t)-\rn'(t+\vT_{\eta_l})=0$ with $t:=\theta-\vT_{\eta_l}$, and consider the zeros of $f$. 
	The assumption \eqref{eq:StarSmlCond1} implies that $f(-\theta_1)<0<f(\theta_1)$ and $f'>0$ in $(-\theta_1,\theta_1)$, $f(\theta_2)>0>f(\pi-\theta_2)$ and $f'<0$ in $(\theta_2,2\pi-\theta_2)$, and $|f|>0$ in $[\theta_1,\theta_2]\cup[2\pi-\theta_{2},2\pi-\theta_1]$.
	Thus there are exactly two zeros of $f$, one located in $(-\theta_1,\theta_1)$, denoted as $t=\Cp_{1,l}\vT_\eta-\vT_{\eta_l}$, and the other in $(\theta_2,2\pi-\theta_2)$, denoted as $t=\Cp_{2,l}\vT_\eta-\vT_{\eta_l}$. The proof is complete.
\end{proof}
\begin{rem}\label{rem:admissible}
	We note that a necessary condition for \eqref{eq:StarSmlCond0} is
	\begin{equation*}
		|\rn'(t)|<\frac{\sqrt{q}}{\sqrt{1-q}}\AND
		-\frac{\sqrt{q}}{1-\sqrt{q}}<\rn''(t)<\frac{\sqrt{q}}{1+\sqrt{q}},\qquad\mbox{for all $t$},
	\end{equation*}
	that is, the assumption \eqref{eq:StarSmlCond0} in the definition of admissible star-shaped domains. On the other hand, if \eqref{eq:StarSmlCond0} is satisfied, then one can find $\theta_1\in(0,\tilde{\theta}_q)$ and $\theta_2\in(\tilde{\theta}_q,\pi)$ such that $\rn'$ satisfies \eqref{eq:StarSmlCond1}.
	In other words, \eqref{eq:StarSmlCond0} is equivalent to \eqref{eq:StarSmlCond1} for some $\theta_1\in(0,\tilde{\theta}_q)$ and $\theta_2\in(\tilde{\theta}_q,\pi)$.
\end{rem}
Under the notations of Lemma~\ref{lem:4criticStarSml}, it is observed that
\begin{equation}\label{eq:CppiStarBi}
	\Cp_{j,l}(\vT_{\eta}+\pi)=\Cp_{j,l+1}\vT_\eta,\qquad
	(-1)^{l-1}\sgn\sin(\Cp_{j,l}\vT_\eta-\vT_{\eta})=\sgn\rn'(\Cp_{j,l}\vT_\eta),\qquad\mbox{for all $\eta$},
\end{equation}
\begin{equation*}\label{eq:cossintStarBi}
\cos(\Cp_{1,l}\vT_\eta-\vT_{\eta_l})>\cos\theta_1>-\sqrt{q}>\cos\theta_2>\cos(\Cp_{2,l}\vT_\eta-\vT_{\eta_l}),\qquad\mbox{for all $\eta$},
\end{equation*}
and
\begin{equation}\label{eq:rho'=0}
	\Cp_{j,j}\vT_{\eta}=\vT_{\eta},\quad j=1,2,
	\qquad\mbox{for all $\eta$ such that $\rho'(\vT_{\eta})=0$.}
\end{equation}
Furthermore, consider \eqref{eq:criticalStar2} again we can derive that
\begin{equation*}
	\begin{split}
		\sqrt{q}\cos(\Cp_{j,l}\vT_\eta-\vT_{\eta_l}) &= \frac{-\rn'^2+(-1)^{j-1}\varrho}{1+\rn'^2}\big|_{\Cp_{j,l}\vT_\eta}
	=\frac{1-q}{1-(-1)^{j-1}\varrho}\big|_{\Cp_{j,l}\vT_\eta}-1,
	\\\sqrt{q}\sin(\Cp_{j,l}\vT_\eta-\vT_{\eta_l})
	&=\rn'\frac{1+(-1)^{j-1}\varrho }{(1+\rn'^2)}\,\big|_{\Cp_{j,l}\vT_\eta}
	=\frac{(1-q)\rn'}{1-(-1)^{j-1}\varrho}\big|_{\Cp_{j,l}\vT_\eta},
	\end{split}
\end{equation*}
where
\begin{equation}\label{eq:Rstar}
	\varrho (t):=\sqrt{q-(1-q)\rn'^2(t)},
	\qquad\mbox{which satisfies $1-\varrho^2=(1-q)(1+\rn'^2)$}.
\end{equation}
Recall from \eqref{eq:HessianStat} we also obtain that
\begin{equation}\label{eq:detD2}
	\frac{\det D^2\psi_\eta|_{(\Cp_{j,l}\vT_{\eta},\Cp_{j,l}\vT_{\eta}+\delta_{l,2}\pi)}}{\rho^2(\Cp_{j,l}\vT_{\eta})} = (-1)^{j-1}
	\Pare{\varrho-(-1)^{j-1}\frac{(1-q)\rn''}{1-(-1)^{j-1}\varrho}}  |_{\Cp_{j,l}\vT_{\eta}}.
\end{equation}
\begin{lem}\label{lem:crimoreStarSml}
	Under the same assumptions and notations as in Lemma~\ref{lem:4criticStarSml}, we have for each $t\in\RR$ and each $l=1,2$ that
	\begin{equation*}
		\sgn\rn'(\Cp_{1,l}t)=\sgn\rn'(\Cp_{2,l+1}t).
	\end{equation*}
	In other words,
	\begin{equation*}
		\sgn\sin(\Cp_{1,l}t-t)=-\sgn\sin(\Cp_{2,l+1}t-t).
	\end{equation*}
\end{lem}
\begin{proof}
	Thanks to the first relation in \eqref{eq:CppiStarBi}, it is sufficient to prove the statement for the case when $l=1$.

	Recall that, $\theta=\Cp_{1,1}t$ is the unique solution to $\rn'(\theta)=h(\theta-t)$ with $\theta-t\in(-\theta_1,\theta_1)$, and $\theta=\Cp_{2,2}t$ the unique solution to $\rn'(\theta)=h(\theta-\vT_{\eta_2})$ with  $\theta-t\in(\theta_2-\pi,\pi-\theta_2)$.

	If $\rn'(\Cp_{2,2}t)=h(\Cp_{2,2}t-\vT_{\eta_2})=0$, then $\sin(\Cp_{2,2}t-t)=0$, and hence $\Cp_{2,2}t=t$ and $\rn'(t)=0$. Consequently, $\Cp_{1,1}t=t$, and so  $\rn'(\Cp_{1,1}t)=0$. 

	Suppose that $\rn'(\Cp_{2,2}t)=h(\Cp_{2,2}t-\vT_{\eta_2})>0$. Then $\sin(\Cp_{2,2}t-\vT_{\eta_2})>0$ and $\Cp_{2,2}t-t\in(\theta_2-\pi,0)$.  We claim that $\Cp_{1,1}t-t\in(0,\theta_1)$, and hence $\rn'(\Cp_{1,1}t)>0$. Otherwise if $\Cp_{1,1}t-t\in(-\theta_1,0]\subset(-\tilde{\theta}_q,0]$, then $\rn'(\Cp_{1,1}t)\le 0<\rn'(\Cp_{2,2}t)$ and thus $\Cp_{1,1}t\neq\Cp_{2,2}t$. If $\Cp_{2,2}t>\Cp_{1,1}t$ (in the interval $(\pi-t,t)$) then
	\begin{equation*}
			\rn'(\Cp_{2,2}t)-\rn'(\Cp_{1,1}t)
			>-h(\Cp_{1,1}t-t)
			>h(\Cp_{2,2}t-t)-h(\Cp_{1,1}t-t),
	\end{equation*}
	and hence $\rn''(t_1)>h'(t_2)=h'(-t_2)$ for some $t_1\in(\Cp_{1,1}t, \Cp_{2,2}t)\subset(\pi-t,t)$ and $t_2\in (\Cp_{1,1}t, \Cp_{2,2}t)-t\subset(-\theta_1,0)$. However, this contradicts \eqref{eq:StarSmlCond1} since
	\begin{equation*}
		\min\{h'(0),h'(\theta_1)\}=\min_{t\in[0,\theta_1]}\{h'(t)\}
		\AND h'(\theta_2)=\max_{t\in[\theta_2,\pi]}\{h'(t)\}.
	\end{equation*}
	Then we can only have $t-\theta_q<\Cp_{2,2}t<\Cp_{1,1}t\le t$. In this case we derive that
	\begin{equation*}
		\rn'(\Cp_{1,1}t)-\rn'(\Cp_{2,2}t)
		\le-h(\Cp_{2,2}t-t+\pi)
		<h(\Cp_{1,1}t-t+\pi)-h(\Cp_{2,2}t-t+\pi),
	\end{equation*}
	and consequently $\rn''(t_1)<h'(t_2)=h'(2\pi-t_2)$ for some $t_1\in(\Cp_{2,2}t,\Cp_{1,1}t)$ and $t_2\in (\Cp_{2,2}t,\Cp_{1,1}t)-t+ \pi\subset(\theta_2,\pi]$. Therefore $\rn''(t_1)<h'(\theta_2)$, which contradicts \eqref{eq:StarSmlCond1} again.

	With analogous arguments we can show that $\rn'(\Cp_{2,2}t)<0$ implies $\Cp_{2,2}t-t\in(0,\pi-\theta_2)$, $\Cp_{1,1}t-t\in(-\theta_1,0)$ and $\rn'(\Cp_{1,1}t)<0$.
\end{proof}

Next, we shall consider the equation \eqref{eq:criticalStar2} from another point of view. Given any $\theta$, define
\begin{equation*}
	\gamma(\theta):=\arcsin\frac{\rn'(\theta)}{\sqrt{1+\rn'^2(\theta)}}\in(-{\pi}/{2},\pi/2).
\end{equation*}
We can then rewrite \eqref{eq:criticalStar2} as
\begin{equation*}
	\sqrt{q}\sin(\theta-\gamma(\theta)-\vT_{\eta_l})=\sin\gamma(\theta).
\end{equation*}
Adapting the notations in Lemma~\ref{lem:4criticStarSml} we obtain that
\begin{equation*}
	\vT_{\eta}=\Pc_{j,l}(\theta):=\theta-\gamma(\theta)-(-1)^{j-1}\arcsin\frac{\sin\gamma(\theta)}{\sqrt{q}}+\delta_{j,l}\pi,\qquad \mbox{where $\theta=\Cp_{j,l}\vT_{\eta}$.}
\end{equation*}
Therefore, $\Pc_{j,l}$ can be deemed as the (left) inverse of $\Cp_{j,l}$.
Direct calculation yields
\begin{equation}\label{eq:Sjl'}
	\gamma'(\theta)=\frac{\rn''(\theta)}{1+\rn'^2(\theta)}
	\AND\Pc_{j,l}'(\theta)
	=1-\frac{(-1)^{j-1}(1-q)\rn''(\theta)}{\varrho(\theta)\Pare{1-(-1)^{j-1}\varrho(\theta)}},
\end{equation}
where the function $\varrho$ is defined in \eqref{eq:Rstar}.
Recalling \eqref{eq:detD2} we observe for any $(\theta,\vT_\xi)=(\theta,\theta+\delta_{l,2}\pi)\in\cS_{\vT_\eta}$ that
\begin{equation*}
	\det D^2\psi_\eta(\theta,\vT_\xi)\neq 0\qquad\mbox{if and only if}\qquad \Pc_{j,l}'(\theta)\neq 0.
\end{equation*}
Moreover, under the condition \eqref{eq:StarSmlCond0} in Theorem~\ref{thm:main1Star}, or \eqref{eq:StarSmlCond1} in Lemma~\ref{lem:4criticStarSml}, we have that
\begin{equation*}
	\Pc_{j,l}'>0\qquad\mbox{in a neighborhood of any $t$ such that $\rho'(t)=0$}.
\end{equation*}
As a consequence,
\begin{equation*}
	\Cp_{j,l}'>0\qquad\mbox{in a neighborhood of any $t$ such that $\rho'(t+\delta_{j,l}\pi)=0$}.
\end{equation*}
If $\Omega$ is an admissible star-shaped domain, then $\Cp_{j,l}$ monotonically increases everywhere. We summarize the result in the following.
\begin{lem}\label{lem:critMapStar}
	Assume that
	\begin{equation*}\label{eq:StarCond2}
		\frac{(-1)^{j-1}(1-q)\rn''(t)}{\varrho (t)\Pare{1-(-1)^{j-1}\varrho (t)}}<1,\qquad
		\mbox{for $j=1,2$ and all $t$}.
	\end{equation*}
	Then for each $j,l=1,2$, $\Cp_{j,l}:\RR/(2\pi\mathbb{Z})\to\RR /\{2\pi\}$ is a $C^1$  strictly increasing bijection.
	Moreover,
	\begin{equation*}
		\det D^2\psi_\eta(\theta,\vT_\xi)\neq 0\qquad\mbox{for all $(\theta,\vT_\xi)=(\Cp_{j,l}\vT_\eta,\Cp_{j,l}\vT_\eta+\delta_{l,2}\pi)$}.
	\end{equation*}
\end{lem}
We deduce from \eqref{eq:asympGen_diffeta} that, for all $\eta\in\Ss^1$,
\begin{equation}\label{eq:IntNasympStar}
	\text{\small$\frac{k\Int_k^{(N)}(\vT_\eta)}{2\pi\pare{ik\sqrt{q}}^N}=
		\sum_{j,l=1}^{2}
		\frac{\Psi_\eta e^{\frac{i\pi}4\sgn D^2\psi_\eta+ik\psi_\eta}}{\abs{\det D^2\psi_\eta}^{1/2}}\big|_{(\Cp_{j,l}\vT_\eta,\Cp_{j,l}\vT_\eta+\delta_{l,2}\pi)}
		\phi\Pare{(-1)^{l-1}\vec\Cp_{j,l}\vT_\eta}f_{j,l}^N(\vT_\eta)
		+O(k^{N-2})$},
\end{equation}
as $k\to \infty$, where
\begin{equation}\label{eq:fjl}
	f_{j,l}(\vT_\eta)=\rho(\Cp_{j,l}\vT_\eta) \sin(\Cp_{j,l}\vT_\eta-\vT_{\eta}).
\end{equation}

We will need the following elementary result concerning $4\times4$ Vandermonde matrices, of which the proof we choose to omit.
\begin{lem}\label{lem:Vander}
	Consider the linear algebraic system
	\begin{equation*}
		\begin{bmatrix}
			1&1&1&1\\a_{1,1}&a_{1,2}&a_{2,1}&a_{2,2}\\a_{1,1}^2&a_{1,2}^2&a_{2,1}^2&a_{2,2}^2\\a_{1,1}^3&a_{1,2}^3&a_{2,1}^3&a_{2,2}^3
		\end{bmatrix}
		\begin{bmatrix}
			z_{1,1}\\z_{1,2}\\z_{2,1}\\z_{2,2}
		\end{bmatrix}
		= 0,
	\end{equation*}
	with $\sgn a_{1,1}=-\sgn a_{2,2}>0$ and $\sgn a_{2,1}=-\sgn a_{1,2}\ge0$. If $a_{1,1}= a_{2,1}$ and $a_{1,2}=a_{2,2}$ then $z_{1,1}+z_{2,1}=z_{1,2}+z_{2,2}=0$. If $a_{1,1}=a_{2,1}$ and $a_{1,2}\neq a_{2,2}$, then $z_{1,1}+z_{2,1}=0$ and $z_{1,2}=z_{2,2}=0$. If $a_{2,1}=0$, then $z_{1,2}+z_{2,1}=0$ and $z_{1,1}=z_{2,2}=0$. If $a_{2,1}\neq0$, $a_{1,1}\neq a_{2,1}$ and $a_{1,2}\neq a_{2,2}$, then $z_{1,1}=z_{2,1}=z_{1,2}=z_{2,2}=0$.
\end{lem}

\subsection{Proof of the Main Results}
In this subsection, given $\phi\in C^1(\Ss^1)$, we sometimes use the notation $\phi(\vT_\xi)$ instead of $\phi(\xi)$, and treat $\phi$ as a function defined on $\RR/ \{2\pi\}$ for convenience.
\subsubsection{Proof of Theorem~\ref{thm:StarAnly}}
Thanks to Lemmas~\ref{lem:starting}~and~\ref{lem:kinfty}, Theorem~\ref{thm:StarAnly} is a direct consequence of the following.
\begin{prop}\label{prop:thm:StarAnly}
	Let $\phi\in C^1(\Ss^1)$ \Blu{be nonzero almost everywhere in $\Ss^1$.} 
	and let $\rho$ be the radius function of an admissible $C^2$ star-shaped domain. Suppose that there exist a sequence of wavenumbers $k$'s tending to infinity such that $\Int_k(\vT_\eta)=0$ for all $\vT_\eta$, where $\Int_k$ is defined in \eqref{eq:IntFinal}. Then $\rho$ must satisfy
	\begin{equation*}
		\sgn\rho'(t)=-\sgn\rho'(t+\pi)\qquad \mbox{for all $t$,}
	\end{equation*}
	and $\rho''(t)=\rho''(t+\pi)=0$ whenever $\rho'(t)=0$.
\end{prop}
\begin{proof}
	We assume (without loss of generality) that
	\begin{equation}\label{eq:prfrhoStar0}
		\mbox{$|\rho'|>0$ \quad in $(0,\tau)$ \AND $\rho'(0)=\rho'(\tau)=0$},
	\end{equation}
	for some $\tau\in(0,2\pi)$. Then by \eqref{eq:CppiStarBi}, \eqref{eq:rho'=0} and Lemma~\ref{lem:critMapStar} we deduce that, $\Cp_{1,1},\Cp_{2,2}$ are strictly increasing $C^1$ bijections onto $(0,\tau)$. Moreover, Lemma~\ref{lem:crimoreStarSml} implies that $\sgn f_{1,1}(\vT_\eta) \sgn f_{2,2}(\vT_\eta)=-1$ for all $\vT_\eta\in(0,\tau)$.
	Similar to the proof of Proposition~\ref{prop:ellip}, regarding \eqref{eq:IntNasympStar}, we can formulate a linear homogeneous algebraic equation with a $4\times 4$ Vandermonde coefficient matrix whose entries are generated by $f_{j,l}^N(\vT_\eta)$, $j,l=1,2$, $N=0,\ldots,3$, where the ``unknowns'' are \begin{equation}\label{eq:VandUnknown}
		\frac{\phi(\xi)\Psi_\eta(\theta,\vT_\xi)\omega_{j,l}}{\abs{\det D^2\psi_\eta(\theta,\vT_\xi)}^{1/2}},
		\qquad\mbox{with $\theta=\Cp_{j,l}\vT_\eta$, $\vT_\xi=\Cp_{j,l}\vT_\eta+\delta_{l,2}\pi$, $\omega_{j,l}\in\CC$, $|\omega_{j,l}|=1$.}
	\end{equation}
	\Blu{Since $\phi$ is nonzero almost everywhere, by} 
	Lemma~\ref{lem:Vander} we then must have $\sgn f_{1,2}(\vT_\eta) \sgn f_{2,1}(\vT_\eta)=-1$ for almost all $\vT_\eta\in(0,\tau)$.
	Combining \eqref{eq:CppiStarBi} we further obtain for almost all $t\in(0,\tau)$ that one of the following cases holds:
	\begin{equation}\label{eq:11212212}
		\begin{split}
			&\sgn\rho'(\Cp_{1,2}t)=\sgn\rho'(\Cp_{2,1}t)=\sgn\rho'(t)
			\\&\rho(\Cp_{1,1}t) \sin(\Cp_{1,1}t-t)=\rho(\Cp_{2,1}t) \sin(\Cp_{2,1}t-t)
			,\quad
			|\phi(\Cp_{1,1}t)|^2=\Gp_{11,21}(t)|\phi(\Cp_{2,1}t)|^2,
			\\&\rho(\Cp_{2,2}t) \sin(\Cp_{2,2}t-t)=\rho(\Cp_{1,2}t) \sin(\Cp_{1,2}t-t)
			,\quad
			|\phi(\Cp_{2,2}t+\pi)|^2=\Gp_{22,12}(t)|\phi(\Cp_{1,2}t-\pi)|^2,
		\end{split}
	\end{equation}
	or
	\begin{equation}\label{eq:11122221}
		\begin{split}
			&\sgn\rho'(\Cp_{1,2}t)=\sgn\rho'(\Cp_{2,1}t)=-\sgn\rho'(t),
			\\&\rho(\Cp_{1,1}t) \sin(\Cp_{1,1}t-t)=\rho(\Cp_{1,2}t) \sin(\Cp_{1,2}t-t)
			,\quad
			|\phi(\Cp_{1,1}t)|^2=\Gp_{11,12}(t)|\phi(\Cp_{1,2}t-\pi)|^2,
			\\&\rho(\Cp_{2,2}t) \sin(\Cp_{2,2}t-t)=\rho(\Cp_{2,1}t) \sin(\Cp_{2,1}t-t)
			,\quad
			|\phi(\Cp_{2,2}t+\pi)|^2=\Gp_{22,21}(t)|\phi(\Cp_{2,1}t)|^2.
		\end{split}
	\end{equation}
	Hereinafter, for $j_1,j_2,l_1,l_2=1,2$,
	\begin{equation*}
		\begin{split}
			\Gp_{j_1l_1,j_2l_2}(t):=&\ABs{\frac{\det D^2\psi_{\vec t}}{\Psi_{\vec t}^2}}_{(\Cp_{j_1,l_1}t, \Cp_{j_1,l_1}t+\delta_{l_1,2}\pi)} \ABs{\frac{\Psi_{\vec t}^2}{\det D^2\psi_{\vec t}}}_{(\Cp_{j_2,l_2}t, \Cp_{j_2,l_2}t+\delta_{l_2,2}\pi)}
			\\=&\frac{\varrho-(-1)^{j_1-1}\frac{(1-q)\rn''}{1-(-1)^{j_1-1}\varrho}}{\Pare{1-(-1)^{j_1-1}\varrho}^2}\big|_{\Cp_{j_1,l_1}t}~~ \frac{\Pare{1-(-1)^{j_2-1}\varrho}^2}{\varrho-(-1)^{j_2-1}\frac{(1-q)\rn''}{1-(-1)^{j_2-1}\varrho}}\big|_{\Cp_{j_2,l_2}t}\neq 0.
		\end{split}
	\end{equation*}

	We prove that either \eqref{eq:11212212} holds for all $t\in(0,\tau)$, or \eqref{eq:11122221} holds for all $t\in(0,\tau)$. Otherwise there exists $t_1\in(0,\tau)$ and $t_2\in[0,\tau]$ such that $t_1\neq t_2$, $\rho'(\Cp_{1,2}t_1)=0$ and $\rho'(\Cp_{1,2}t)\neq0$ for all $t$ in between $t_1$ and $t_2$, say, all $t\in(t_1,t_2)$. Then either \eqref{eq:11212212} holds for all $t\in(t_1,t_2)$, or \eqref{eq:11122221} holds for all $t\in(t_1,t_2)$. Hence by continuity either \eqref{eq:11212212} or \eqref{eq:11122221} holds for $t=t_1$, which implies that $\sin(\Cp_{j,j}t_1-t_1)=\sin(\Cp_{1,2}t_1-t_1)=0$ with $j=1$ or $2$. Thus we obtain that $\rho'(\Cp_{j,j}t_1)=0$, a contradiction to \eqref{eq:prfrhoStar0}.

	By continuity we then observe that either \eqref{eq:11212212} or \eqref{eq:11122221} is satisfied for all $t\in[0,\tau]$. In particular,
	\begin{equation*}\label{eq:prfrho2Star}
		\mbox{$|\rho'|>0$ \quad in $(\pi,\tau+\pi)$ \AND $\rho'(\pi)=\rho'(\tau+\pi)=0$}.
	\end{equation*}
	Therefore, we obtain from \eqref{eq:CppiStarBi} and Lemma~\ref{lem:critMapStar} that $\Cp_{1,1},\Cp_{2,2}:[0,\tau]\to[0,\tau]$, $\tilde{\Cp}_{1,1},\tilde{\Cp}_{2,2}:[\pi,\tau+\pi]\to[\pi,\tau+\pi]$, $\Cp_{2,1},\Cp_{1,2}:[0,\tau]\to[\pi,\tau+\pi]$, and $\tilde{\Cp}_{2,1},\tilde{\Cp}_{1,2}:[\pi,\tau+\pi]\to[0,\tau]$ are all monotonically increasing bijections. Here, $\tilde{\Cp}_{j,l}$ is the same mapping as  $\Cp_{j,l}$ but with the domain restricted on the interval $[\pi,\tau+\pi]$.

	Next we assume that $\rho'>0$ in $(0,\tau)$, and hence $\rho''(0)\ge 0\ge\rho''(\tau)$.
	We discuss the two different cases \eqref{eq:11212212} and \eqref{eq:11122221}, or equivalently, when $\rho'>0$ in $(\pi,\tau+\pi)$ and  when $\rho'<0$ in $(\pi,\tau+\pi)$. We show that the first case cannot occur.
	Assume otherwise that $\rho'>0$ in $(0,\tau)+\pi$, and hence $\rho''(\pi)\ge0\ge\rho''(\tau+\pi)$. In this case \eqref{eq:11212212} is satisfied  for all $t\in[0,\tau]$, and with an analogous argument one can also show that it satisfied for all $t\in[\pi,\tau+\pi]$.
	Direct calculation yields
	\begin{equation*}
		\Gp_{j_1l_1,j_2l_2}(t)= \frac{\Pare{1-(-1)^{j_2-1}\sqrt{q}}^2}{\Pare{1-(-1)^{j_1-1}\sqrt{q}}^2}~ \frac{\sqrt{q}-(-1)^{j_1-1}\Pare{1+(-1)^{j_1-1}\sqrt{q}}\rn''(\Cp_{j_1,l_1}t)}{\sqrt{q}-(-1)^{j_2-1}\Pare{1+(-1)^{j_2-1}\sqrt{q}}\rn''(\Cp_{j_2,l_2}t)},
	\end{equation*}
	for any $t$ such that $\rho'(\Cp_{j_1,l_1}t)=\rho'(\Cp_{j_2,l_2}t)=0$.
	Then
	\begin{equation}\label{eq:prfG1121>1}
		\Gp_{11,21}(\tau)=\frac{(1+\sqrt{q})^2}{(1-\sqrt{q})^2} \frac{\sqrt{q}-(1+\sqrt{q})\rn''(\tau)}{\sqrt{q}+(1-\sqrt{q})\rn''(\tau+\pi)}>1,
	\end{equation}
	and similarly $\Gp_{11,21}(\tau+\pi)>1$. Substituting $t=\tau$ and $t=\tau+\pi$ in \eqref{eq:11212212} then yields $$\phi(\tau)=\phi(\pi+\tau)=0.$$ In addition, we obtain that $\Gp_{11,21}(t)>1$ for $t\in[t_1,\tau]\cup [t_1+\pi,\tau+\pi]$ with some $t_1\in(0,\tau)$. Therefore by \eqref{eq:11212212},
	\begin{equation}\label{eq:prf11ge21}
		|\phi(\Cp_{1,1}t)|^2\ge|\phi(\Cp_{2,1}t)|^2,\qquad t\in[t_1,\tau]\cup [t_1,\tau]+\pi,
	\end{equation}
	where the equal sign holds if and only if $\phi(\Cp_{1,1}t)=\phi(\Cp_{2,1}t)=0$. Since $\phi$ is \Blu{$C^1$ and nonzero almost everywhere,} 
	there exists $t_2\in[t_1,\tau)$ such that
	\begin{equation}\label{eq:prfNormStar}
		|\phi(t_2)|=\|\phi\|_{C^0[t_1,\tau]}>0\AND |\phi(t)|<|\phi(t_2)|\quad\mbox{for all $t\in(t_2,\tau)$}.
	\end{equation}
	Let $t_0=\Cp_{2,1}^{-1}\tilde\Cp_{1,1}\tilde\Cp_{2,1}^{-1}t_2$. Since $\rho'>0$ in $(0,\tau)\cup(\pi,\tau+\pi)$, by \eqref{eq:CppiStarBi} we obtain that
	\begin{equation*}
		0<\tilde\Cp_{2,1}(t+\pi)<t<\Cp_{1,1}t<\tau<\Cp_{2,1}t<t+\pi<\tilde\Cp_{1,1}(t+\pi)<\tau+\pi,\qquad t\in(0,\tau).
	\end{equation*}
	As a consequence, $\tilde\Cp_{2,1}^{-1}t_2\in(t_2+\pi,\tau+\pi)$, $\tilde{\Cp}_{1,1}\tilde\Cp_{2,1}^{-1}t_2-\pi\in(t_2,\tau)$,  $t_0=\Cp_{2,1}^{-1}\tilde{\Cp}_{1,1}\tilde\Cp_{2,1}^{-1}t_2\in(t_2,\tau)$, and $\Cp_{1,1} t_0\in(t_2,\tau)$.
	Applying \eqref{eq:prf11ge21} with $t=t_0\in(t_2,\tau)$ and $t=\tilde\Cp_{2,1}^{-1}t_2\in(t_2,\tau)+\pi$ yields
	\begin{equation*}
		|\phi(\Cp_{1,1}t_0)|^2\ge|\phi(\Cp_{2,1}t_0)|^2=|\phi(\tilde{\Cp}_{1,1}\tilde\Cp_{2,1}^{-1}t_2)|^2\ge \phi(t_2)|^2,
	\end{equation*}
	which contradicts \eqref{eq:prfNormStar}.

	Up to now, we have proven that if
	\begin{equation}\label{eq:prfrho>0}
		\mbox{$\rho'>0$ \quad in $(0,\tau)$ \AND $\rho'(0)=\rho'(\tau)=0$},
	\end{equation}
	then
	\begin{equation}\label{eq:prfrho<0}
		\mbox{$\rho'<0$ \quad in $(0,\tau)+\pi$ \AND $\rho'(\pi)=\rho'(\tau+\pi)=0$},
	\end{equation}
	and \eqref{eq:11122221} is satisfied for all $t\in[0,\tau]$.
	Notice that
	\begin{equation*}
		\Gp_{11,12}(0)=\frac{\sqrt{q}-(1+\sqrt{q})\rn''(0)}{\sqrt{q}-(1+\sqrt{q})\rn''(\pi)}\le1\le \frac{1}{\Gp_{11,12}(0)}=\Gp_{11,12}(\pi),
	\end{equation*}
	where in the inequality we have made use of the fact that $\rho''(\pi)\le 0\le\rho''(0)$.
	Suppose that $|\rho''(0)|+|\rho''(\pi)|\neq 0$, and thus $\rho''(0)>\rho''(\pi)$. Then $\Gp_{11,12}(0)<1$ and hence by \eqref{eq:11122221} that $\phi(0)=0$. In addition, $\Gp_{11,12}(t)>1$ for $t\in[-t_1,t_1]$ with some $t_1\in(0,\tau)$. Similar as before, we can find $t_2\in(0,t_1]\subset(0,\tau)$ such that
	\begin{equation*}
		|\phi(t_2)|=\|\phi\|_{C^0[0,t_1]}>0\AND |\phi(t)|<|\phi(t_2)|\quad\mbox{for all $t\in(0,t_2)$}.
	\end{equation*}
	Since $\rho'>0$ in $(0,\tau)$ and $\rho'<0$ in $(\pi,\tau+\pi)$, \eqref{eq:CppiStarBi} implies that
	\begin{equation*}
		0<\Cp_{1,2}t-\pi<t<\Cp_{1,1}t<\tau,\qquad t\in(0,\tau),
	\end{equation*}
	and hence $0<\Cp_{1,2}\Cp_{1,1}^{-1}t_2-\pi<\Cp_{1,1}^{-1}t_2<t_2$. We apply \eqref{eq:11122221} for $t=\Cp_{1,1}^{-1}t_2\in(0,t_2)$ and obtain
	\begin{equation*}
		|\phi(t_2)|^2<|\phi(\Cp_{1,2}\Cp_{1,1}^{-1}t_2-\pi)|^2<|\phi(t_2)|^2,
	\end{equation*}
	which is a contradiction. Therefore, $\rho$ must satisfy $\rho''(0)=\rho''(\pi)=0$. With an analogous argument, by considering $\Gp_{11,12}(\tau)$, one can also prove that $\rho''(\tau)=\rho''(\pi+\tau)=0$.

	Similarly, it can be shown that \eqref{eq:prfrho<0} implies \eqref{eq:prfrho>0}. We choose to omit the details. Furthermore, for any $t$ such that $\rho'(t)=0$ we must have $\rho'(t+\pi)=0$; Otherwise $\rho'(t+\pi)\neq0$ would imply $\rho'(t)\neq 0$ as shown before. So far, we have proven that $\sgn\rho'(t)=-\sgn\rho'(t+\pi)$ for all $t$, and that, if $t$ is a zero of $\rho'$ such that $|\rho'|>0$ in a neighborhood of $t_+$ or of $t_-$, then $\rho'(t)=\rho'(t+\pi)=\rho''(t)=\rho''(t+\pi)=0$. It is evident that if $\rho'=0$ in some interval $(t_1,t_2)$, then $\rho'=\rho''=0$ in $[t_1,t_2]\cup[t_1,t_2]+\pi$. We are left to show $\rho'(t_0)=0$ implies $\rho''(t_0)=\rho''(t_0+\pi)=0$ when $t_0$ is an accumulated zero of $\rho'$ with $\rho'$ not identically zero in a neighborhood of $t_{0,+}$. In this case there exists a sequence of intervals $\{(t_{m,1},t_{m_2})\}$ satisfying $\rho'\neq 0$ in $(t_{m,1},t_{m_2})$, $\rho'(t_{m,1})=\rho'(t_{m,2})=0$, $t_{m_1}>t_0$, and $t_{m_1}\to t_0$ as $m\to\infty$. Then, as proven before, $\rho''(t_{m,1})=\rho''(t_{m,1}+\pi)=0$, and hence by continuity that $\rho''(t_0)=\rho''(t_0+\pi)=0$.

	The proof is complete.
\end{proof}

\subsubsection{Proof of Theorem~\ref{thm:main1Star}}
\begin{prop}
	Let $\phi\in C^1(\Ss^1)$, and let $\rho$ be the radius function of an admissible $C^2$ star-shaped domain.
	Suppose that there exists $\N=\cup_\iota(s_{\iota,1},s_{\iota,2})$, a union of countable (finite or infinite) disjoint open intervals of $(0,2\pi)+\theta_0$ for some $\theta_0\in(0,2\pi)$, such that, as subsets of the quotient space $\RR/(2\pi\mathbb{Z})$, $\N$ and $\N+\pi$ 
	satisfy $\overline{\N\cup(\N+\pi)}=\RR/(2\pi\mathbb{Z})$. Moreover, for each $\iota$,
	\begin{equation*}
		\mbox{$|\rho'|>0$ in $(s_{\iota,1},s_{\iota,2})$,\qquad $\rho'(s_{\iota,1})=\rho'(s_{\iota,2})=\rho'(s_{\iota,1}+\pi)=\rho'(s_{\iota,2}+\pi)=0$,}
	\end{equation*}
	and one of the followings is satisfied:
	\begin{enumerate}[(i)]
		\item $\sgn\rho'(t+\pi)=\sgn\rho'(t)$ for $t\in(s_{\iota,1},s_{\iota,2})$.
		\item $\rho''=0$ in either $(s_{\iota,1}+\pi,s_{\iota,1}+\pi+\epsilon)$ or $(s_{\iota,2}+\pi-\epsilon,s_{\iota,2}+\pi)$ for some $\epsilon>0$.
		\item $|\rho''(s_{\iota,j})|+|\rho''(s_{\iota,j}+\pi)|> 0$, for some $j=1,2$.
	\end{enumerate}

	If there exist a sequence of wavenumbers $k$'s tending to infinity such that $\Int_k(t)=0$ for all $t$, where $\Int_k$ is defined in \eqref{eq:IntFinal}, then $\phi$ must be identically zero.
\end{prop}
\begin{proof}
	We assume without loss of generality that $(0,s_0)$ is an interval from $\N$ such that
	\begin{equation}\label{eq:prfrho'>0}
		\mbox{$|\rho'|>0$ \quad in $(0,s_0)$ \AND $\rho'(0)=\rho'(s_0)=\rho'(\pi)=\rho'(s_0+\pi)=0$.}
	\end{equation}
	Then \eqref{eq:CppiStarBi} and Lemma~\ref{lem:critMapStar} imply that $\Cp_{1,1},\Cp_{2,2}:[0,s_0]\to[0,s_0]$, $\tilde{\Cp}_{1,1},\tilde{\Cp}_{2,2}:[0,s_0]+\pi\to[0,s_0]+\pi$, $\Cp_{2,1},\Cp_{1,2}:[0,s_0]\to[0,s_0]+\pi$, and $\tilde{\Cp}_{2,1},\tilde{\Cp}_{1,2}:[0,s_0]+\pi\to[0,s_0]$ are all monotonically increasing bijections. Here,
	\begin{equation}\label{eq:tTjl}
		\tilde{\Cp}_{j,l}(t+\pi):=\Cp_{j,l+1}t,\qquad t\in[0,s_0],\quad j,l=1,2.
	\end{equation}
	In other words, recalling \eqref{eq:CppiStarBi}, $\tilde{\Cp}_{j,l}$ is nothing but $\Cp_{j,l}$ with its domain restricted on the interval $[0,s_0]+\pi$.

	Note that it suffices to conclude the proof by showing
	\begin{equation*}
		\mbox{$\phi=0$\qquad in\quad $(0,s_0)\cup(\pi,s_0+\pi)$,}
	\end{equation*}
	since the interval $(0,s_0)$ can be replaced with any $(s_{\iota,1},s_{\iota,2})$ from $\N$ satisfying \eqref{eq:prfrho'>0}.
	\subsubsection*{Case (i)}
	Assume that $\rho'>0$ in $(0,s_0)\cup(\pi,s_0+\pi)$. Then by \eqref{eq:CppiStarBi} we have
	\begin{equation}\label{eq:Cp22112112}
		0<\Cp_{2,2}t<t<\Cp_{1,1}t<s_0\AND 0<\Cp_{2,1}t-\pi<t<\Cp_{1,2}t-\pi<s_0,\qquad t\in(0,s_0).
	\end{equation}

	We first show that $\phi(s_0)=0$. Otherwise $\phi$ is nowhere zero near $s_0$, and hence so is $\phi\circ\Cp_{1,1}$. From Lemma~\ref{lem:Vander} we obtain that
	\begin{equation}\label{eq:1121}
		\rho(\Cp_{1,1}t) \sin(\Cp_{1,1}t-t)=\rho(\Cp_{2,1}t) \sin(\Cp_{2,1}t-t)
		,\quad
		|\phi(\Cp_{1,1}t)|^2=\Gp_{11,21}(t)|\phi(\Cp_{2,1}t)|^2,
	\end{equation}
	is satisfied in $(s_0-\epsilon,s_0]$ for some $\epsilon>0$. Taking $t=s_0$ in \eqref{eq:1121} we observe that $\phi(s_0+\pi)\neq0$. Thus \eqref{eq:1121} is also satisfied in $(s_0-\epsilon,s_0]+\pi$, with a possible smaller value of $\epsilon$. Similar to \eqref{eq:prfG1121>1} we deduce that $\Gp_{11,21}(s_0)>1$ and $\Gp_{11,21}(s_0+\pi)>1$. Substituting $t=s_0$ and $t=s_0+\pi$ in \eqref{eq:1121} then yields $\phi(s_0)=\phi(s_0+\pi)=0$, a contradiction.

	Notice that $\Gp_{11,21}(t)>1$ for $t\in[t_1,s_0]\cup [t_1+\pi,s_0+\pi]$ with some $t_1\in(0,s_0)$, and the second identity in \eqref{eq:1121} is valid whether $\phi(\Cp_{1,1}t)$ is zero or not. Therefore,
	\begin{equation*}
		|\phi(\Cp_{1,1}t)|^2\ge|\phi(\Cp_{2,1}t)|^2,\qquad t\in[t_1,s_0]\cup [t_1,s_0]+\pi.
	\end{equation*}
	Suppose that $\phi$ is not identically zero in $(t_1,s_0)$. Then there exists $t_2\in[t_1,s_0)$ such that
	\begin{equation*}
		|\phi(t_2)|=\|\phi\|_{C^0[t_1,s_0]}>0\AND |\phi(t)|<|\phi(t_2)|\quad\mbox{for all $t\in(t_2,s_0)$}.
	\end{equation*}
	Applying the same arguments as in the proof of Proposition~\ref{prop:thm:StarAnly}, we can derive a contradiction to the above statement. Therefore, we conclude that $\phi\equiv0$ in $(t_1,s_0)$. Applying the second identity in \eqref{eq:1121} yields $\phi\equiv0$ in $\Cp_{2,1}\Cp_{1,1}^{-1}(t_1,s_0)$. Repeating this procedure once again we have $\phi\equiv0$ in $\mathcal{U}(t_1,s_0)$, where $\mathcal{U}:=\tilde\Cp_{2,1}\tilde\Cp_{1,1}^{-1}\Cp_{2,1}\Cp_{1,1}^{-1}$. By induction we obtain that $\phi\equiv0$ in $\mathcal{U}^m(t_1,s_0)$ for $m\in\mathbb{N}$. Notice from \eqref{eq:Cp22112112} that $\mathcal{U}$ is a monotonically decreasing bijection onto $[0,s_0]$, with the only fix points being $0$ and $s_0$. Therefore, $\mathcal{U}^mt_1\to 0$ as $m\to\infty$ and thus $\phi\equiv0$ in $[0,s_0]$.

	The case when $\rho'<0$ in $(0,s_0)\cup(\pi,s_0+\pi)$ can be prove analogously by conduction the arguments near $0$ and $\pi$. We opt to omit the details.

	\subsubsection*{Cases (ii) and (iii)}
	Let $\inT_1:=\{t\in(0,s_0):\rho'(t+\pi)\neq 0\}$ and $\inT_0:=[0,s_0]\backslash\inT_1$. Then $\tilde{\Cp}_{1,1},\tilde{\Cp}_{2,2}:\inT_0+\pi\to\inT_0+\pi$ and $\Cp_{1,2}-\pi,\Cp_{2,1}-\pi:\inT_0\to\inT_0$ are the identity operators. Same as in the proof of Proposition~\ref{prop:thm:StarAnly}, from \eqref{eq:IntNasympStar} we can construct a $4\times 4$ linear homogeneous algebraic system with ``unknowns'' as in \eqref{eq:VandUnknown} and a Vandermonde coefficient matrix generated from $f_{j,l}(t)$, $j,l=1,2$, where $f_{j,l}$ is defined in \eqref{eq:fjl}.
	Notice that
	\begin{equation*}
		\begin{split}
			&f_{1,1}=f_{2,2}=0\quad \text{in } \inT_0+\pi,
			\qquad f_{1,2}=f_{2,1}=0\quad \text{in }  \inT_0,
			\\&\sgn f_{1,1}=-\sgn f_{2,2}\neq0\quad \text{in } (0,s_0)\cup(\inT_1+\pi),
			\\& \sgn f_{1,2}=-\sgn f_{2,1}\neq0\quad \text{in }  \inT_1\cup((0,s_0)+\pi).
		\end{split}
	\end{equation*}
	We then deduce from Lemma~\ref{lem:Vander} that,
	\begin{equation*}
		\begin{split}
			&\phi(\Cp_{1,1}t)=\phi(\Cp_{2,2}t+\pi)=0\AnD
			|\phi(\Cp_{1,2}t-\pi)|^2=\Gp_{12,21}(t)|\phi(\Cp_{2,1}t)|^2,\qquad t\in \inT_0,
			\\&\phi(\tilde\Cp_{1,2}t+\pi)=\phi(\tilde\Cp_{2,1}t)=0\AnD
			|\phi(\tilde\Cp_{1,1}t)|^2=\Gp_{11,22}(t)|\phi(\tilde\Cp_{2,2}t-\pi)|^2,\qquad t\in \inT_0+\pi.
		\end{split}
	\end{equation*}
	Thus, recalling \eqref{eq:tTjl}, we obtain
	\begin{equation*}
		\phi(\Cp_{1,1}t)=\phi(\Cp_{1,1}t+\pi)=\phi(\Cp_{2,2}t)=\phi(\Cp_{2,2}t+\pi)=0,\qquad t\in \inT_0.
	\end{equation*}
	In other words,
	\begin{equation}\label{eq:prfphi1122=0}
		\phi=0\qquad\mbox{in  $\Cp_{1,1}\inT_0\cup\Cp_{2,2}\inT_0\cup\pare{(\Cp_{1,1}\inT_0\cup\Cp_{2,2}\inT_0)+\pi}$}.
	\end{equation}

	Define
	\begin{equation*}
		\begin{split}
		&	\mbox{$\inT_{1,+}:=\{t\in(0,s_0):\sgn\rho'(t+\pi)=\sgn\rho'(t)\neq 0\}\subset \inT_1$, \qquad and,}
			\\&\mbox{$\inT_{1,-}:=\{t\in(0,s_0):\sgn\rho'(t+\pi)=-\sgn\rho'(t)\neq 0\}=\inT_1\backslash\inT_{1,+}$.}
		\end{split}
	\end{equation*}
	Then by \eqref{eq:CppiStarBi},
	\begin{equation*}
		\begin{split}
			&\sgn f_{1,1}=\sgn f_{2,1}=-\sgn f_{1,2}=-\sgn f_{2,2}\neq0\quad \text{in } \inT_{1,+}\cup(\inT_{1,+}+\pi),
			\\& \sgn f_{1,1}=\sgn f_{1,2}=-\sgn f_{2,1}=-\sgn f_{2,2}\neq0\quad \text{in }  \inT_{1,-}\cup(\inT_{1,-}+\pi).
		\end{split}
	\end{equation*}
	Therefore, for each $t\in\inT_{1,+}\cup(\inT_{1,+}+\pi)$, we obtain from Lemma~\ref{lem:Vander} that
	\begin{itemize}
		\item $\phi(\Cp_{1,1}t)$ and $\phi(\Cp_{2,1}t)$ are either both zero or both nonzero, and in the \eqref{eq:1121} is valid.
		\item $\phi(\Cp_{2,2}t+\pi)$ and $\phi(\Cp_{1,2}t-\pi)\neq0$ are either both zero or both nonzero, and in the latter case,
		\begin{equation*}
			\rho(\Cp_{2,2}t) \sin(\Cp_{2,2}t-t)=\rho(\Cp_{1,2}t) \sin(\Cp_{1,2}t-t)
			,\quad
			|\phi(\Cp_{2,2}t+\pi)|^2=\Gp_{22,12}(t)|\phi(\Cp_{1,2}t-\pi)|^2.
		\end{equation*}
	\end{itemize}
	Moreover, recalling \eqref{eq:tTjl}, we can reformulate this statement in the set $\inT_{1,+}+\pi$ to one in $\inT_{1,+}$; that is, for each $t\in\inT_{1,+}$,
	\begin{itemize}
		\item $\phi(\Cp_{1,2}t)$ and $\phi(\Cp_{2,2}t)$ are either both zero or both nonzero, and in the latter case we have
		\begin{equation}\label{eq:1222}
			\rho(\Cp_{1,2}t) \sin(\Cp_{1,2}t-t)=\rho(\Cp_{2,2}t) \sin(\Cp_{2,2}t-t)
			,\quad
			|\phi(\Cp_{1,2}t)|^2=\Gp_{12,22}(t)|\phi(\Cp_{2,2}t)|^2;
		\end{equation}
		\item $\phi(\Cp_{2,1}t-\pi)$ and $\phi(\Cp_{1,1}t+\pi)\neq0$ are either both zero or both nonzero, and in the latter case,
		\begin{equation}\label{eq:2111}
			\rho(\Cp_{2,1}t) \sin(\Cp_{2,1}t-t)=\rho(\Cp_{1,1}t) \sin(\Cp_{1,1}t-t)
			,\quad
			|\phi(\Cp_{2,1}t-\pi)|^2=\Gp_{21,11}(t)|\phi(\Cp_{1,1}t+\pi)|^2.
		\end{equation}
	\end{itemize}
	In particular,
	\begin{equation}\label{eq:phijl>0}
		\begin{split}
		&|\phi(\Cp_{1,1}t)|^2=\Gp_{11,21}(t)|\phi(\Cp_{2,1}t)|^2,\quad
		|\phi(\Cp_{2,2}t+\pi)|^2=\Gp_{22,12}(t)|\phi(\Cp_{1,2}t-\pi)|^2,
		\\&|\phi(\Cp_{1,2}t)|^2=\Gp_{12,22}(t)|\phi(\Cp_{2,2}t)|^2,\quad
		|\phi(\Cp_{2,1}t-\pi)|^2=\Gp_{21,11}(t)|\phi(\Cp_{1,1}t+\pi)|^2,
		\end{split}
		\qquad t\in\inT_{1,+}.
	\end{equation}
	Analogously, one can also show
	for each $t\in\inT_{1,-}$ that
	\begin{itemize}
		\item $\phi(\Cp_{1,1}t)$ and $\phi(\Cp_{1,2}t-\pi)$ are both or neither zero, and in the latter case,
		\begin{equation}\label{eq:1112}
			\rho(\Cp_{1,1}t) \sin(\Cp_{1,1}t-t)=\rho(\Cp_{1,2}t) \sin(\Cp_{1,2}t-t)
			,\quad
			|\phi(\Cp_{1,1}t)|^2=\Gp_{11,12}(t)|\phi(\Cp_{1,2}t-\pi)|^2;
		\end{equation}
		\item $\phi(\Cp_{2,2}t+\pi)$ and $\phi(\Cp_{2,1}t)$ are both or neither zero, and in the latter case,
		\begin{equation}\label{eq:2221}
			\rho(\Cp_{2,2}t) \sin(\Cp_{2,2}t-t)=\rho(\Cp_{2,1}t) \sin(\Cp_{2,1}t-t)
			,\quad
			|\phi(\Cp_{2,2}t+\pi)|^2=\Gp_{22,21}(t)|\phi(\Cp_{2,1}t)|^2;
		\end{equation}
		\item $\phi(\Cp_{1,2}t)$ and $\phi(\Cp_{1,1}t+\pi)$ are both or neither zero, and in the latter case,
		\begin{equation}\label{eq:1211}
			\rho(\Cp_{1,2}t) \sin(\Cp_{1,2}t-t)=\rho(\Cp_{1,1}t) \sin(\Cp_{1,1}t-t)
			,\quad
			|\phi(\Cp_{1,2}t)|^2=\Gp_{12,11}(t)|\phi(\Cp_{1,1}t+\pi)|^2;
		\end{equation}
		\item $\phi(\Cp_{2,1}t-\pi)$ and $\phi(\Cp_{2,2}t)$ are both or neither zero, and in the latter case,
		\begin{equation}\label{eq:2122}
			\rho(\Cp_{2,1}t) \sin(\Cp_{2,1}t-t)=\rho(\Cp_{2,2}t) \sin(\Cp_{2,2}t-t)
			,\quad
			|\phi(\Cp_{2,1}t-\pi)|^2=\Gp_{21,22}(t)|\phi(\Cp_{2,2}t)|^2;
		\end{equation}
	\end{itemize}
	and in particular,
	\begin{equation}\label{eq:phijl<0}
		\begin{split}
			&|\phi(\Cp_{1,1}t)|^2=\Gp_{11,12}(t)|\phi(\Cp_{1,2}t-\pi)|^2,\quad
			|\phi(\Cp_{2,2}t+\pi)|^2=\Gp_{22,21}(t)|\phi(\Cp_{2,1}t)|^2,
			\\&|\phi(\Cp_{1,2}t)|^2=\Gp_{12,11}(t)|\phi(\Cp_{1,1}t+\pi)|^2,\quad
			|\phi(\Cp_{2,1}t-\pi)|^2=\Gp_{21,22}(t)|\phi(\Cp_{2,2}t)|^2.
		\end{split}
	\end{equation}

	If $\inT_0=[0,s_0]$, then we readily obtain from \eqref{eq:prfphi1122=0} that $\phi=0$ in $(0,s_0)\cup(\pi,s_0+\pi)$. We assume $\inT_0\neq[0,s_0]$ in the rest, namely, $\inT_1\neq\emptyset$.
	We also assume that
	\begin{equation}\label{eq:rho'>0}
		\rho'>0\qquad\mbox{in $(0,s_0)$}.
	\end{equation}
	Then by \eqref{eq:CppiStarBi} we have
	\begin{equation}\label{eq:T22<11}
		0<\Cp_{2,2}t<t<\Cp_{1,1}t<s_0,\qquad t\in(0,s_0).
	\end{equation}
	{In the case when $\rho'<0$ in $(0,s_0)$, all the arguments in the rest of the proof can be proceeded analogously by switching the roles of $\Cp_{1,1}$ and $\Cp_{2,1}$ with the roles of $\Cp_{2,2}$ and $\Cp_{1,2}$, respectively.}

	Assume $\phi\nid0$ in $(0,s_0)$. Then
	\begin{equation*}
		\mbox{$\phi(t)\neq 0$ for $t\in[t_1,t_2]$ with $0\le t_1<t_2\le s_0$}
	\end{equation*}
	that is, $\phi(\Cp_{1,1}t)\neq 0$ for $t\in\Cp_{1,1}^{-1}[t_1,t_2]\subset(0,s_0)$. We first observe from \eqref{eq:prfphi1122=0} that $\Cp_{1,1}^{-1}[t_1,t_2]\subset\inT_1$. Moreover, by the definition of $\inT_1$, we must have
	\begin{equation}\label{eq:proofI+-}
		\mbox{either\quad $\Cp_{1,1}^{-1}[t_1,t_2]\subset\inT_{1,+}$,\qquad or\quad $\Cp_{1,1}^{-1}[t_1,t_2]\subset\inT_{1,-}$.}
	\end{equation}
	Suppose that $\Cp_{1,1}^{-1}[t_1,t_2]\subset(\tau_1,\tau_2)\subset\inT_{1,+}$ with
	\begin{equation*}
		\mbox{$\rho'>0$\quad in $(\tau_1,\tau_2)+\pi$\AND $\rho'(\tau_1+\pi)=\rho'(\tau_2+\pi)=0$}.
	\end{equation*}
	Then by \eqref{eq:CppiStarBi},
	\begin{equation}\label{eq:T21<12}
		\tau_1<\Cp_{2,1}t-\pi<t<\Cp_{1,2}t-\pi<\tau_2,\qquad t\in(\tau_1,\tau_2).
	\end{equation}
	Since $\phi(\Cp_{1,1}t)\neq 0$ for $t\in\Cp_{1,1}^{-1}[t_1,t_2]\subset\inT_{1,+}$, we deduce 
	from \eqref{eq:phijl>0} that
	\begin{equation*}
		\mbox{$\phi$ is nowhere zero in $\CP_1[t_1,t_2]+\pi$,}
	\end{equation*}
	where $\CP_1:=\Cp_{2,1}\Cp_{1,1}^{-1}-\pi$. Moreover,

	Thus $\phi\circ(\Cp_{1,1}+\pi)$ is nowhere zero in $\Cp_{1,1}^{-1}\CP_1[t_1,t_2]$. By \eqref{eq:prfphi1122=0} again we must have $\CP_1[t_1,t_2]\cap\inT_0=\emptyset$ and hence
	\begin{equation*}
		\mbox{either\quad $\Cp_{1,1}^{-1}\CP_1[t_1,t_2]\subset\inT_{1,+}$,\qquad or\quad $\Cp_{1,1}^{-1}\CP_1[t_1,t_2]\subset\inT_{1,-}$.}
	\end{equation*}
	If $\Cp_{1,1}^{-1}[t_1,t_2]\subset(\tau_1,\tau_2)\subset\inT_{1,-}$ with
	\begin{equation*}
		\mbox{$\rho'<0$\quad in $(\tau_1,\tau_2)+\pi$\AND $\rho'(\tau_1+\pi)=\rho'(\tau_2+\pi)=0$}.
	\end{equation*}
	Then by \eqref{eq:CppiStarBi},
	\begin{equation}\label{eq:T21>12}
		\tau_1<\Cp_{1,2}t-\pi<t<\Cp_{2,1}t-\pi<\tau_2,\qquad t\in(\tau_1,\tau_2).
	\end{equation}
	Since $\phi\circ\Cp_{1,1}$ is nowhere zero in $\Cp_{1,1}^{-1}[t_1,t_2]$, we deduce from \eqref{eq:phijl<0} that \begin{equation*}
		\mbox{$\phi$ is nowhere zero in $\CP_2[t_1,t_2]$,}
	\end{equation*}
	where $\CP_2:=\Cp_{1,2}\Cp_{1,1}^{-1}-\pi$. Thus $\phi\circ\Cp_{1,1}$ is nowhere zero in $\Cp_{1,1}^{-1}\CP_2[t_1,t_2]$. By \eqref{eq:prfphi1122=0} again we must have $\CP_2[t_1,t_2]\cap\inT_0=\emptyset$ and hence
	\begin{equation*}
		\mbox{either\quad $\Cp_{1,1}^{-1}\CP_2[t_1,t_2]\subset\inT_{1,+}$,\qquad or\quad $\Cp_{1,1}^{-1}\CP_2[t_1,t_2]\subset\inT_{1,-}$.}
	\end{equation*}

	Denote $\kappa_1=0$ and $[t_{1,1},t_{1,2}]:=[t_1,t_2]$. With an inductive argument we can show there exist a sequence of binary indices $\{\iota_m:m\in\NN_+\}$ and $\{\kappa_m:m\in\NN_+\}$, and a sequence of intervals $\{[t_{m,1},t_{m,2}]\subset(0,s_0):m\in\NN_+\}$ with $t_{m+1,1}<t_{m+1,2}$, such that for every $m\in\NN_+$,
	\begin{equation}\label{eq:prfphineq0}
		\mbox{$\phi$ is nowhere zero in $[t_{m,1},t_{m,2}]+\delta_{\kappa_m,1}\pi$},
	\end{equation}
	and
	\begin{equation*}
		\mbox{either\quad $\Cp_{1,1}^{-1}[t_{m,1},t_{m,2}]\subset\inT_{1,+}$ and $\iota_{m}=1$,\qquad
		or\quad $\Cp_{1,1}^{-1}[t_{m,1},t_{m,2}]\subset\inT_{1,-}$ and $\iota_{m}=2$,}
	\end{equation*}
	with
	\begin{equation*}
		[t_{m+1,1},t_{m+1,2}]:=\CP_{\iota_m}[t_{m,1},t_{m,2}]
		\AND \kappa_{m+1}=\kappa_m+\iota_m,\qquad m\in\NN_+.
	\end{equation*}
	In addition, notice from \eqref{eq:T22<11}, \eqref{eq:T21<12}, and \eqref{eq:T21>12} that $\CP_{\iota_m}$ satisfies
	\begin{equation}\label{eq:Umdecreas}
		0<\CP_{\iota_m}t<\Cp_{1,1}^{-1}t<t<s_0,\qquad t\in\Cp_{1,1}^{-1}[t_{m-1,1},t_{m-1,2}]\subset(0,s_0).
	\end{equation}
	Moreover, the only points in $[0,s_0]$ such that $\Cp_{1,1}t=t$ are $t=0$ and $t=s_0$. Thus, we deduce that
	\begin{equation}\label{eq:prooft-0}
		t_{m,1}\to 0\qquad\mbox{as $m\to\infty$.}
	\end{equation}
	and
	\begin{equation}\label{eq:proofNonid0}
		\mbox{$\inT_1\cap(0,\epsilon)\neq\emptyset$ \qquad for any $\epsilon>0$.}
	\end{equation}

	\subsubsection*{Case (ii)}
	If $\rho''=0$ in a neighborhood of $0_+$, then we have already reached a contradiction to \eqref{eq:proofNonid0}. If $\rho''=0$ in $(s_0-\epsilon_0,s_0)$ for some $\epsilon_0>0$, we can conduct analogous arguments as before but interchanging the roles of $\Cp_{1,1}$, $\Cp_{1,2}$, and $\Cp_{2,1}$, with respectively, $\Cp_{2,2}$, $\Cp_{2,1}$, and $\Cp_{1,2}$. In particular, the definitions of $\CP_1$ and $\CP_2$ will be $\CP_1=\Cp_{1,2}\Cp_{2,2}^{-1}-\pi$ and $\CP_2=\Cp_{2,1}\Cp_{2,2}^{-1}-\pi$. Thus \eqref{eq:Umdecreas} will be replaced with
	\begin{equation*}
		0<t<\Cp_{2,2}^{-1}t<\CP_{\iota_m}t<t<s_0,\qquad t\in\Cp_{2,2}^{-1}[t_{m-1,1},t_{m-1,2}]\subset(0,s_0).
	\end{equation*}
	As a consequence, $t_{m,2}\to s_0$ as $m\to\infty$ and $\inT_1\cap(s_0-\epsilon,s_0)\neq\emptyset$ for any $\epsilon>0$, a contradiction.
	Therefore, we have concluded for Case (ii) that $\phi=0$ in $(0,s_0)$. With similar arguments one can show that $\phi=0$ in $(\pi,s_0+\pi)$. We choose to skip the details.

	\subsubsection*{Case (iii)}
	We first consider the case when the condition in Case (iii) is satisfied with $j=1$; that is, $\rho''(0)$ and $\rho''(\pi)$ are not both zero.
	Recalling \eqref{eq:prfphineq0} and \eqref{eq:prooft-0} we obtain that $\phi$ is not identically zero in any neighborhood of $0_+$ or $\pi_+$ (or both). Suppose that $\phi\not\equiv 0$ in any neighborhood of $0_+$. (The case when $\phi\not\equiv 0$ near $\pi_-$ can be handled analogously and we choose to omit the details.) Then for any $m\in\NN$, there exists a nonempty interval $(\alpha_{m,1},\alpha_{m,2})\subset(0,1/m)$ where $\phi$ is nowhere zero. In other words, $\phi(\Cp_{1,1}t)\neq 0$ for $t\in\Cp_{1,1}^{-1}(\alpha_{m,1},\alpha_{m,2})$. Hence \eqref{eq:proofI+-} holds when $[t_1,t_2]$ is replaced with each $(\alpha_{m,1},\alpha_{m,2})$. Moreover, if $\Cp_{1,1}^{-1}(\alpha_{m,1},\alpha_{m,2})\subset\inT_{1,+}$, then \eqref{eq:1121} is satisfied in $\Cp_{1,1}^{-1}(\alpha_{m,1},\alpha_{m,2})$; otherwise \eqref{eq:1112} is satisfied in the same interval. In particular, either
	\begin{equation}\label{eq:1121d}
		d\rho(\Cp_{1,1}t) \sin(\Cp_{1,1}t-t)/dt=d\rho(\Cp_{2,1}t) \sin(\Cp_{2,1}t-t)/dt,
	\end{equation}
	or
	\begin{equation}\label{eq:1112d}
		d\rho(\Cp_{1,1}t) \sin(\Cp_{1,1}t-t)/dt=d\rho(\Cp_{1,2}t) \sin(\Cp_{1,2}t-t)/dt,
	\end{equation}
	is satisfied in each $\Cp_{1,1}^{-1}(\alpha_{m,1},\alpha_{m,2})$, and hence by continuity, at $t=0$.
	In addition, since $\phi(\Cp_{2,2}t)\neq 0$ for $t\in\Cp_{2,2}^{-1}(\alpha_{m,1},\alpha_{m,2})$, similar as before, we can deduce that either \eqref{eq:1222} is satisfied in $\Cp_{2,2}^{-1}(\alpha_{m,1},\alpha_{m,2})\subset\inT_{1,+}$, or \eqref{eq:2122} is satisfied in  $\Cp_{2,2}^{-1}(\alpha_{m,1},\alpha_{m,2})\subset\inT_{1,-}$. As a consequence, either
	\begin{equation}\label{eq:1222d}
		d\rho(\Cp_{1,2}t) \sin(\Cp_{1,2}t-t)/dt=d\rho(\Cp_{2,2}t) \sin(\Cp_{2,2}t-t)/dt,
	\end{equation}
	or
	\begin{equation}\label{eq:2122d}
		d\rho(\Cp_{2,1}t) \sin(\Cp_{2,1}t-t)/dt=d\rho(\Cp_{2,2}t) \sin(\Cp_{2,2}t-t)/dt,
	\end{equation}
	holds at $t=0$.

Recalling \eqref{eq:Sjl'} we have for $t=0$ that
\begin{equation*}
	1/\Cp_{j,l}'(t)=1-(-1)^{j-1}\Pare{1+(-1)^{j-1}\sqrt{q}}\rn''(t+\delta_{j,l+1}\pi)/\sqrt{q}.
\end{equation*}
Hence for $t=0$,
\begin{equation*}
	\begin{split}
		d\rho(\Cp_{j,l}t)& \sin(\Cp_{j,l}t-t)/dt
		=\rho'(\Cp_{j,l}t) \sin(\Cp_{j,l}t-t)\Cp_{j,l}'t+\rho(\Cp_{j,l}t) \cos(\Cp_{j,l}t-t)(\Cp_{j,l}'t-1)
		\\&=(-1)^{j-l}\rho(t+\delta_{j,l+1}\pi)(\Cp_{j,l}'t-1)
		\\&=(-1)^{j-l}\rho(t+\delta_{j,l+1}\pi)
		\frac{(-1)^{j-1}\Pare{1+(-1)^{j-1}\sqrt{q}}\rn''(t+\delta_{j,l+1}\pi)/\sqrt{q}}{1-(-1)^{j-1}\Pare{1+(-1)^{j-1}\sqrt{q}}\rn''(t+\delta_{j,l+1}\pi)/\sqrt{q}}.
	\end{split}
\end{equation*}
By the assumptions for Case (iii), $\rho''(0)$ and $\rho''(\pi)$ are not both zero.
If \eqref{eq:1121d} holds at $t=0$, then  $\rho''(0)\rho''(\pi)\neq0$ and
\begin{equation}\label{eq:1121d0}
	\frac{1}{\rho(0)}
	\Pare{\frac{\sqrt{q}}{\Pare{1+\sqrt{q}}\rn''(0)}-1}
	=\frac{1}{\rho(\pi)}
	\Pare{\frac{\sqrt{q}}{\Pare{1-\sqrt{q}}\rn''(\pi)}+1}.
\end{equation}
Moreover, notice that $\rho''(0)>0$ due to \eqref{eq:rho'>0}. Combining Remark~\ref{rem:admissible} we deduce that $\rho''(\pi)>0$.
However, this contradicts \eqref{eq:2122d}, which reads
\begin{equation}\label{eq:2122d0}
	\frac{1}{\rho(0)}
	\Pare{\frac{\sqrt{q}}{\Pare{1-\sqrt{q}}\rn''(0)}+1}
	=-\frac{1}{\rho(\pi)}
	\Pare{\frac{\sqrt{q}}{\Pare{1-\sqrt{q}}\rn''(\pi)}+1},
\end{equation}
with the LHS positive and RHS negative.
Furthermore, \eqref{eq:1121d0} contradicts \eqref{eq:1222d} at $t=0$ as well. In fact, the latter implies
\begin{equation}\label{eq:1222d0}
	\frac{1}{\rho(0)}
	\Pare{\frac{\sqrt{q}}{\Pare{1-\sqrt{q}}\rn''(0)}+1}
	=\frac{1}{\rho(\pi)}
	\Pare{\frac{\sqrt{q}}{\Pare{1+\sqrt{q}}\rn''(\pi)}-1}.
\end{equation}
Combining \eqref{eq:1121d0} we obtain that
\begin{equation*}
	\rho''(0)+\rho''(\pi)=\frac{\sqrt{q}}{1+\sqrt{q}}-\frac{\sqrt{q}}{1-\sqrt{q}}<0,
\end{equation*}
a contradiction.
If otherwise \eqref{eq:1112d} is satisfied at $t=0$, then
\begin{equation*}
	\frac{1}{\rho(0)}
	\Pare{\frac{\sqrt{q}}{\Pare{1+\sqrt{q}}\rn''(0)}-1}
	=-\frac{1}{\rho(\pi)}
	\Pare{\frac{\sqrt{q}}{\Pare{1+\sqrt{q}}\rn''(\pi)}-1},
\end{equation*}
and hence, $\rho''(0)>0>\rho''(\pi)$. But this again is not compatible with either \eqref{eq:2122d0} or \eqref{eq:1222d0}.

We are left to prove $\phi\equiv 0$ in $(0,s_0)$ when $\rho''(s_0)$ and $\rho''(s_0+\pi)$ are not both zero. Similar to Case (ii), by interchanging the roles of $\Cp_{1,1}$, $\Cp_{1,2}$, and $\Cp_{2,1}$, with respectively, $\Cp_{2,2}$, $\Cp_{2,1}$, and $\Cp_{1,2}$, we can obtain, with analogous arguments for deriving \eqref{eq:prfphineq0} that $\phi$ is not identically either near $s_0$ or near $s_0+\pi$. In the former case, we can deduce that either \eqref{eq:1121d} or \eqref{eq:1112d} holds at $t=s_0$, and either \eqref{eq:1222d} or \eqref{eq:2122d} is satisfied at $t=s_0$. A similar contradiction will follow as before. The case when $\phi$ is not identically near $s_0+\pi$ can be dealt with analogously and we choose to omit the technical details.

The proof is complete.
\end{proof}

\section*{Acknowledgments}
The work of J. Xiao is partially supported by NSF grant DMS-23-07737.

\addcontentsline{toc}{section}{References}
\bibliographystyle{abbrv}%
\bibliography{VXEllipse}

\end{document}